\documentclass[11pt,a4j,twoside]{article}

\usepackage{amsmath,amssymb}
\usepackage{amsthm}
\usepackage[abbrev]{amsrefs}
\usepackage{latexsym}

\usepackage{graphicx}

\usepackage{caption}
\usepackage{float}

\usepackage[english]{babel}
\usepackage{fancyhdr}
\usepackage{a4wide}

\usepackage{bm}
\usepackage{algorithm}
\usepackage{algpseudocode}

\allowdisplaybreaks[1]

\numberwithin{equation}{section}

\newtheorem{thm}{Theorem}[section]
\newtheorem{prop}[thm]{Proposition}
\newtheorem{lem}[thm]{Lemma}

\newtheorem{dfn}[thm]{Definition}
\theoremstyle{definition} 
\newtheorem{ex}[thm]{Example}
\newtheorem{rem}[thm]{Remark}

\newcommand\ND{\newcommand}

\ND\lref[1]{Lemma~\ref{#1}}
\ND\tref[1]{Theorem~\ref{#1}}
\ND\pref[1]{Proposition~\ref{#1}}
\ND\sref[1]{Section~\ref{#1}}
\ND\ssref[1]{Subsection~\ref{#1}}
\ND\aref[1]{Appendix~\ref{#1}}
\ND\rref[1]{Remark~\ref{#1}}
\ND\cref[1]{Corollary~\ref{#1}}
\ND\eref[1]{Example~\ref{#1}}
\ND\fref[1]{Fig.\ {#1} }
\ND\lsref[1]{Lemmas~\ref{#1}}
\ND\tsref[1]{Theorems~\ref{#1}}
\ND\dref[1]{Definition~\ref{#1}}
\ND\psref[1]{Propositions~\ref{#1}}
\ND\rsref[1]{Remarks~\ref{#1}}
\ND\sssref[1]{Subsections~\ref{#1}}
\ND\esref[1]{Examples~\ref{#1}}
\ND\asref[1]{Assumption~\ref{#1}}

\newcommand{\ep}{\varepsilon}

\newcommand{\h}{\quad}
\newcommand{\dis}{\displaystyle}

\newcommand{\barD}{\overline{D}}

\newcommand{\bn}{\mathbf{n}}

\newcommand{\cV}{\mathcal{V}}
\newcommand{\cJ}{\mathcal{J}}

\newcommand{\cM}{\mathcal{M}}
\newcommand{\cD}{\mathcal{D}}
\newcommand{\cP}{\mathcal{P}}
\newcommand{\cW}{\mathcal{W}}

\newcommand{\cS}{\mathcal{S}}
\newcommand{\cT}{\mathcal{T}}
\newcommand{\clD}{\overline{D}}

\title{Topology optimization concerning the mass distribution via filtered gradient flows on the Wasserstein space}
\date{}
\author{Fumiya Okazaki\thanks{Department of School of Computing, Institute of Science Tokyo, 2 Chome-12-1 Ookayama, Meguro-ku, Tokyo 152-8552, Japan\\
Email address: \texttt{okazaki.f.660b@m.isct.ac.jp}}\and
Takayuki Yamada\thanks{Graduate School of Engineering, The University of Tokyo, Yayoi 2-11-16, Bunkyo-ku, Tokyo 113-8656, Japan\\
Email address: \texttt{t.yamada@mech.t.u-tokyo.ac.jp}}}
\begin{document}
\maketitle
\renewcommand{\thepage}{\arabic{page}}
\begin{abstract}
In this article, we formulate topology optimization problems concerning the mass distribution as minimization problems for functionals on the Wasserstein space. We relax optimization problems regarding non-convex objective functions on the Wasserstein space by using the Neumann heat semigroup and prove the existence of minimizers of relaxed problems. Furthermore, we introduce the filtered Wasserstein gradient flow and derive the error estimate between the original Wasserstein gradient flow and the filtered one in terms of the Wasserstein distance. We also construct a candidate for the optimal mass distribution for a given fixed total mass and simultaneously obtain the shape of the material by the numerical calculation of filtered Wasserstein gradient flows.
\end{abstract}

\begin{flushleft}
{\bf MSC2020 Subject Classifications: 35Q93}
\end{flushleft}

\section{Introduction}
Structural optimization is widely used in structural design and studied from the viewpoints of both theoretical aspects and applications. Structural optimization is a problem of finding candidates of shapes of the material which minimize or maximize a certain performance. In the mathematical setting, it is formulated as an optimization problem for objective functions defined on a set of shapes with some constraints. Shape optimization and topology optimization are representatives of structural optimization and in this article, we distinguish these two terms as follows: Shape optimization refers to the problem to optimize the boundary of the shape of the material. On the other hand, topology optimization refers to the optimization problem on the space of more general sets to obtain the distribution of the material. In this sense, topology optimization deals with the most general setting regarding the space of shapes.

Recently, shape optimization has been studied from the viewpoint of optimization problems on infinite dimensional Riemannian manifolds. In \cites{MM05, MM06, MM07}, the geometry of the space of shapes defined by the quotient of embeddings of circles into the plane has been studied through several Riemannian structures introduced on those spaces. Based on those Riemannian structures, it was observed that shape optimization can be regarded as optimization problems on the shape space in \cite{Schulz14}. From this viewpoint, shape optimization has been further studied in e.g. \cites{SSW15, SSW16, Welker21}. In particular, the relation between the shape derivative in the context of the shape calculus and the gradient with respect to those metrics has been clarified.

The aim of this article is to formulate topology optimization as optimization problems on a certain kind of infinite dimensional Riemannian manifold. Since the infinite dimensional Riemannian structures employed in the previous studies mentioned above have restricted their scope to the space that consists of embeddings of a certain fixed manifold, we need to adopt another shape space in order to incorporate the change of topology such as the creation and disappearance of holes.

First of all, the space of Borel subsets in the design domain or equivalently, the space of characteristic functions of Borel subsets, is inconvenient as a space to formulate the topology optimization. In fact, topology optimization problems are essentially ill-posed, which means objective functions do not attain their minima on the space. It is due to the fact that objective functions can decrease as the subset forms the finer microstructure. In order to make the problem well-posed, the relaxation of the problem based on the homogenization method has been proposed in \cite{BK88}, which also provides the methodology to obtain a final shape with different topology from the initial state. Moreover, in \cites{Bendsoe89, BS99, YCL94, RZB92}, the simple isotropic material with penalization, which is called SIMP method in short, has been introduced. SIMP method consists in the simplified version of the homogenization method and lets us replace characteristic functions with density functions $\rho(x)$ and obtain a desired shape by the penalization by setting $p>0$ as the power for the density $\rho$. However, even if we replace characteristic functions with density functions with values in $[0,1]$, it still seems to be hard to guarantee that typical objective functions such as the mean compliance of elastic structures attain their minima in general unless the power for density functions is set to $p=1$. In order to make SIMP method viable, filtering procedures are often employed. There are several types of filtering for densities and sensitivities. For instance, in \cite{Bourdin01}, filtering by the convolution has been theoretically studied in a general setting and the existence of the minimum of the relaxed problem has been shown.

In our study, we focus on the problem to obtain the optimal distribution $\rho(x)$ of mass of the material being subject to $\dis \int \rho = M$ for a given total mass $M>0$. We call such problems mass optimization following \cite{BB01}. Then we formulate the problem as the optimization problem on the space of probability measures called the Wasserstein space. The Wasserstein space is a set of probability measures equipped with the distance function constructed from the transportation cost. It is well-known that the Wasserstein space formally admits a structure of an infinite dimensional Riemannian manifold due to the Otto calculus \cite{Villani03}. Namely, we can define the notion of tangent spaces, the Riemannian metric, the gradient and gradient flows on the Wasserstein space. While conventional topology optimization methods rely on the $L^2$-metric on function spaces, it seems to be more natural to regard a space of shapes as a manifold as observed in previous studies on shape optimization since the formal Riemannian structure on the Wasserstein space is based on the cost of the horizontal movement of distributions rather than the variation of values of density functions.

Formulations of mass optimization on the space of probability measures have already been investigated. In \cites{BBP97, BB01}, PDE-constrained mass optimization problems for the heat conductance and the mean compliance of linear elastic structures were considered and the optimality condition has been derived. In addition, the relation between the $L^1$-Monge-Kantorovich problem and mass optimization problems has been clarified. 
See also \cite{Ambrosio00} for details.

The main results of this article lie in the formulation and procedure of mass optimization on the Wasserstein space. While our formulation of mass optimization is based on the setting in \cite{BB01}, we directly construct flows of the mass distribution and apply the theory of gradient flows on the Wasserstein space. In addition, unlike \cite{BB01}, we impose penalization on density functions following the SIMP method. Due to the power of the density, objective functions are not convex in our setting. Therefore, we relax optimization problems via the convolution with the Neumann heat semigroup in a similar way to \cite{Bourdin01}. In \tsref{wellposed1} and \ref{wellposed2}, we establish the existence of solutions to those relaxed problems. We also show that this relaxation approximates the minimizers of the original objective functions in \tsref{approx1} and \ref{wellposed2}. The proofs of these results can be obtained by simple calculations using the H-convergence theory (e.g. \cite{Allaire02}). In \pref{Sensitivity}, we also observe that the Neumann heat kernel allows us to derive the Wasserstein gradient of objective functions. Methods for updating the mass distribution using Wasserstein gradient flows have been proposed in \cite{OkaYama25} and sensitivities of objective functions have been derived in several settings. However, the approach in \cite{OkaYama25} required the base space to be the whole Euclidean space and demanded the $H^2$-regularity of physical states to derive the sensitivity. We can overcome those obstacles via regularization using the Neumann heat kernel. Nevertheless, in order to obtain suitable candidates of optimal distributions, we need additional filtering for the sensitivities. We therefore introduce a filtered Wasserstein gradient flow with a parameter and provide error estimates in \tref{ErrorEstimate} in terms of the Wasserstein distance between the original gradient flow and the filtered one. Finally, we present numerical examples using the filtered Wasserstein gradient flow. We observe that even though we deal with mass optimization problems and our density functions are not restricted to values in $[0,1]$, optimal distributions obtained numerically in \sref{results} provide implications of optimal shapes.

Other than the previous research mentioned above, the phase field method based on the Cahn-Hilliard equation can also be regarded as the formulation of topology optimization on the space of probability measures. The phase field method is the method to describe the distribution of the material by a phase field function. A phase field function is a function defined on the design domain which is supposed to take its value around $1$ at points where the material exists and around $0$ on the void space. Therefore, we can relax the problem for raw subsets and replace it with the problem for the phase field function. In order to penalize the phase field function in such a way that the values of the phase field function keep separated into $0$ and $1$ in the process of optimization, a double-well potential has been added to the objective function. Then the phase field function is updated so that the so-called extended Ginzburg-Landau-type energy decreases. The Cahn-Hilliard equation is derived as the $H^{-1}$-gradient flow of the Ginzburg-Landau-type energy and applied to topology optimization in \cites{WangZhou04, Zhou05, ZhouWang06, ZhouWang07}. Contrary to the Allen-Cahn equation, which is derived as the $L^2$-gradient flow of the Ginzburg-Landau-type energy and applied to topology optimization in \cites{AC79, TNK10}, the Cahn-Hilliard equation intrinsically preserves the integral of the solution. Therefore, the Cahn-Hilliard equation can also be regarded as the flow on the space of probability measures and can be applied to the topology optimization with the total volume constraint without incorporating the Lagrange multiplier. The main differences of our formulation in this article from the method of the Cahn-Hilliard equation are regularization and penalization. In contrast to the Cahn-Hilliard equation, which is a 4th order nonlinear parabolic PDE, we do not need to consider such a high order PDE in the procedure of optimization since we focus on the mass constraint rather than the volume constraint and employ simpler regularization and penalization. Therefore, in the numerical calculation in \sref{results}, we can assume that the physical state is in the 1st order finite element space and execute the numerical calculation by the finite element method.

The outline of this paper is as follows. In \sref{transport}, we recall some basic settings and facts regarding the theory of the Wasserstein space. In most parts of this section, we refer to \cite{AGS21}. In \sref{Formulation} and following Subsections \ref{minheat} and \ref{mincomp}, we provide the formulation of two typical cases of mass optimization as optimization problems on the Wasserstein space. Then we show that the relaxed problems are well-posed and minimizers of the relaxed problems approximate solutions of the original problems in some sense. In \sref{numerics}, we derive sensitivities of objective functions and introduce the filtering procedure for the sensitivities. Then we show the error estimate regarding the filtered gradient flow. In \sref{results}, we exhibit some numerical examples of mass optimization. 

\section{Preliminary}\label{transport}
In this section, we briefly recall some basic notions regarding the Wasserstein space. Let $d \in \mathbb{N}$ and $D \subset \mathbb{R}^d$ a bounded convex domain. Denote the set of Borel probability measures on $\clD$ by $\cP(\clD)$. For a Borel measurable map $\Phi \colon \clD \to \clD$ and $\mu \in \cP(\clD)$, we denote the push-forward measure by $\Phi_{\sharp}\mu$, which is defined by
\[
\Phi_{\sharp}\mu (A):= \mu(\Phi^{-1}(A))\ \text{for}\ A\in \mathcal{B}(\clD).
\]
For $\mu, \nu \in \mathcal{P}(\clD)$, let $\mathcal{C}(\mu,\nu)$ be the set of couplings of $\mu$ and $\nu$, namely,
\[
\mathcal{C}(\mu,\nu):=\{ \pi \in \mathcal{P}(\clD \times \clD) \mid P_{1 \sharp}\pi=\mu,\ P_{2 \sharp}\pi=\nu \},
\]
where $P_i \colon \clD \times \clD \to \clD$ ($i=1,2$) is the projection to the $i$-th $D$. For $\mu, \nu \in \cP(\clD)$ and $p\in [1,\infty)$, the optimal transport problem for the $p$-th power of the Euclidean metric on $\clD$ in the sense of Kantorovich is to find
\[
I_p(\mu, \nu) := \inf \left\{ \int_{\clD \times \clD} |x-y|^p \, \pi(dxdy) \mid \pi \in \mathcal{C}(\mu, \nu) \right\}
\]
and the minimizer $\pi \in \mathcal{C}(\mu,\nu)$, which is shown to exist in the general theory. The minimizer is called the optimal coupling.
Define the $L^p$-Wasserstein distance $\mathcal{W}_p$ on $\mathcal{P}(\clD)$ by
\[
\mathcal{W}_p(\mu, \nu):= I_p(\mu, \nu)^{\frac{1}{p}}\ \text{for}\ \mu,\nu \in \cP(\clD).
\]
It is known that $\mathcal{W}_p$ is actually a distance function on $\cP(\clD)$. In this paper, we call the pair $(\cP(\clD), \cW_2)$ the $L^2$-Wasserstein space on $\clD$ or simply the Wasserstein space. Since $\clD$ is compact, the convergence of any sequence $\{\mu_n\}_{n\in \mathbb{N}}$ in $(\cP(\clD),\cW_2)$ to some $\mu$ is equivalent to the narrow convergence, that is,
\[
\lim_{n \to \infty}\int_{\clD}\phi\, d\mu_n = \int_{\clD} \phi \, d\mu
\]
for each bounded continuous function $\phi$ on $\clD$.

Next we recall a kind of infinite dimensional Riemannian structure on the Wasserstein space established by Otto \cite{Otto01}.
We say a curve $\{ \mu_t\}_{t\in [0,T]}$ on $(\cP(\clD),\cW_2)$ is absolutely continuous if there exists $f\in L^1([0,T])$ such that
\[
\mathcal{W}_2(\mu_s,\mu_t) \leq \int_s^tf(r)\, dr
\]
for all $s,t \in [0,T]$. If $\{ \mu_t\}_{t\in [0,T]}$ is absolutely continuous, for a.e. $t\in [0,T]$ the curve $\mu_t$ admits the metric derivative $|\mu'_t| \in L^1([0,T])$ defined by
\[
|\mu'_t|:= \lim_{\ep \to 0} \frac{\mathcal{W}_2(\mu_t,\mu_{t+\ep})}{|\ep|}.
\]
Absolute continuous curves on the Wasserstein space can be characterized through the continuity equation. This fact is described in \cite{AGS21} in detail for the case of the Wasserstein space $(\cP_2(\mathbb{R}^d),\cW_2)$ over the whole Euclidean space. Since we now assume that $\clD$ is a closed convex set in $\mathbb{R}^d$, a curve $\mu_t$ in $\cP(\clD)$ is absolutely continuous in $(\cP(\clD),\cW_2)$ if and only if it is absolutely continuous in $(\cP_2(\mathbb{R}^d),\cW_2)$. Therefore, absolutely continuous curves on $(\cP(\clD),\cW_2)$ can also be characterized through the continuity equation in $\cP(\clD)$. For a curve of probability measures $\{\mu_t\}_{t\in [0,T]}$ on $\cP(\clD)$ and a time-dependent Borel vector field $\theta_t\colon \clD \to \mathbb{R}^d$, the pair $(\mu_t,\theta_t)$ is said to satisfy the continuity equation in the distributional sense if
\begin{align}\label{ContiEq}
\int_0^T\int_{\clD} \left(\partial_t\phi(t,x) +\nabla_x \phi(t,x) \cdot \theta_t(x) \right)\, d\mu_t dt=0.
\end{align}
for all $\phi \in C^{\infty}_0((0,T)\times \clD)$, where
\[
C^{\infty}_0((0,T)\times \clD):=C^{\infty}_0((0,T)\times \mathbb{R}^d)|_{(0,T)\times \clD}.
\]
This is equivalent to that $t \to \langle \mu_t, \phi \rangle$ is absolutely continuous and
\[
\frac{d}{dt}\langle \mu_t, \phi \rangle = \int_{\clD} \theta_t \cdot \nabla \phi \, \mu_t(dx)
\]
for a.e. $t \in [0,T]$. Note that this formulation of the continuity equation includes the boundary condition on $\partial D$ since test functions do not necessarily take the value zero near $\partial D$. Due to Theorem 8.3.1 in \cite{AGS21}, if a curve $\mu_t$ is absolutely continuous in $(\cP(\clD),\cW_2)$ with
\begin{align}
\int_0^T|\mu'_t|\, dt<\infty,\label{integrability}
\end{align}
then there exists a Borel vector field $\theta \colon [0,T] \times \clD \to \mathbb{R}^d$ such that $\mu_t$ satisfies the continuity equation \eqref{ContiEq} with respect to $\theta$ in the distributional sense. 


Note that the uniqueness of velocity vector fields satisfying the continuity equation is not guaranteed for every absolutely continuous curve $\mu_t$ without any conditions for velocity vector fields. Indeed, for a pair $(\mu_t,\theta_t)$ satisfying the continuity equation, if we take a vector field $w_t$ on $\clD$ with $\mathrm{div}\mu_t w_t=0$, the pair $(\mu_t,\theta_t + w_t)$ also satisfies the continuity equation. However, if restricted to vector fields which satisfy the equality
\begin{align}\label{vecres}
\| \theta_t\|_{L^2(\mu_t)} \leq |\mu'_t|
\end{align}
for a.e. $t\in [0,T]$, the velocity vector field is uniquely determined by the convexity of $\| \cdot \|_{L^2(\mu_t)}$. 
The unique vector field is determined as follows: For a.e. $t \in [0,T]$, we can define the functional $\partial_t \mu_t$ on $C^{\infty}(\clD)$ by
\[
\langle \partial_t \mu_t, \phi \rangle := \frac{d}{dt} \langle \mu_t, \phi \rangle.
\]
Then by the absolute continuity of $\mu_t$, the functional $\partial_t \mu_t$ satisfies
\begin{align}\label{mutbound}
|\langle \partial_t \mu_t, \phi \rangle| \leq \| \nabla \phi \|_{L^2(\mu_t)} |\mu'_t|. 
\end{align}
See the proof of Theorem 8.3.1 in \cite{AGS21} for details. The inequality \eqref{mutbound} implies that the functional $\partial_t \mu_t$ is bounded linear functional on the weighted homogeneous Sobolev space $\dot{H}^1_{\mu_t}(\clD)$ defined by the completion of the quotient space $C^{\infty}(\clD)/\sim$ with respect to the weighted homogeneous Sobolev seminorm $\| \phi \|_{\dot{H}^1_{\mu}(\clD)}=\|\nabla \phi \|_{L^2(\mu)}$, where the equivalence relation $\sim$ is defined on $C^{\infty}(\clD)$ in such a way that for $\phi,\psi \in C^{\infty}(\clD)$, $\psi \sim \phi$ if and only if $\nabla \phi=\nabla \psi \ \mu$-a.e. Note that the space $\dot{H}^1_{\mu}(\clD)$ is isometric to
\[
\cV_{\mu}:=\overline{\{ \nabla \phi \mid \phi \in C^{\infty}(\clD) \}}^{L^2(\mu)}.
\]
Therefore, by the Riesz representation theorem, there exists $\theta_t \in \cV_{\mu_t}$ such that
\begin{align}\label{Riesz}
\langle \partial_t\mu_t, \psi \rangle_{\dot{H}^{-1}_{\mu_t}(\clD)} = \int_{\clD} \theta_t \cdot \nabla \psi \, d\mu_t
\end{align}
for all $\psi \in C^{\infty}(\clD)$ and a.e. $t\in [0,T]$. This implies that the pair $(\mu_t, \theta_t)$ satisfies the continuity equation in the distributional sense. Moreover, we can deduce from \eqref{mutbound} that the vector field $\theta_t$ satisfies \eqref{vecres}. Therefore, this $\theta_t$ is the velocity vector field associated with $\mu_t$ which is uniquely determined.
Through this procedure, for each $\mu \in \cP(\clD)$, we can regard the negative weighted homogeneous Sobolev space $\dot{H}^{-1}_{\mu}(\clD)$ as the tangent space $T_{\mu}\cP(\clD)$ at $\mu$. Based on the tangent space defined in this way, the gradient of functionals on the Wasserstein space over $\clD$ can be defined as follows.
\begin{dfn}
Let $J \colon \cP(\clD) \to \mathbb{R}$. A tangent vector $\nabla^{\cW}J[\mu] \in T_{\mu}\cP(\clD)$ is called the Wasserstein gradient at $\mu$ if for any absolutely continuous curve $\mu_t$ on $\cP(\clD)$ satisfying the continuity equation with a velocity vector field $\nabla \phi$ with $\phi \in C^{\infty}(\clD)$ near $t=0$, it holds that
\[
\left(\frac{d}{dt}\right)_{t=0}J(\mu_t)=\langle \nabla^{\cW}J[\mu], \phi \rangle_{\dot{H}^{-1}_{\mu_t}(\clD)}.
\]
\end{dfn}


\section{Formulation of mass optimization}\label{Formulation}
We provide the foundation for optimization on $\cP(\clD)$ compatible with mass optimization. Continuing from the previous section, we let $D$ be a bounded convex domain in $\mathbb{R}^d$. Let $\{p_t\}_{t\geq 0}$ be the Neumann heat semigroup on $D$, namely, the map $p_t \colon \mathcal{B}_+(\clD) \to \mathcal{B}_+(\clD)$ is defined by
\[
p_th(x):=\int_D p^r_t(x,y)h(y)\, dy,
\]
where $\mathcal{B}_+(\clD)$ is the set of non-negative Borel functions on $D$ and $p^r_t(x,y)$ is the heat kernel of reflective Brownian motion on $D$. We denote
\[
p_t(x,A):=p_t\mathbf{1}_A (x)
\]
for $A \in \mathcal{B}(\clD)$. Then since reflective Brownian motion is conservative, the Neumann heat semigroup satisfies
\[
p_t(x,D)=p_t1 =1
\]
for any $t>0$. For any fixed $\delta>0$, we set
\begin{align}
c_1&=c_1(\delta):=\inf_{D\times D}p_{\delta}^r(x,y),\label{infHK}\\
c_2&=c_2(\delta):=\sup_{D\times D}p_{\delta}^r(x,y).\label{supHK}
\end{align}
We set
\[
\cD_{\delta}:=\{ \rho(x) dx \in \cP(\clD)\mid c_1 \leq \rho \leq c_2 \}.
\]
Then the Neumann heat semigroup $p_{\delta}$ induces a map from $\cP(\clD)$ to $\cD_{\delta}$ as follows. For $\mu \in \cP(\clD)$, we define $\mu p_{\delta} \in \cP(\clD)$ by
\[
\mu p_{\delta}(A)=\int_D p_{\delta}(x,A)\, \mu(dx),\ A \in \mathcal{B}(\clD).
\]
Then $\mu p_{\delta} \in \cD_{\delta}$ for any $\mu \in \cP(\clD)$. Indeed, it holds that
\begin{align*}
\mu p_{\delta}(A)= \int_A \mu p^r_{\delta}(x)\, dx,
\end{align*}
where
\begin{align*}
\mu p^r_{\delta}(x):=\int_D p^r_{\delta}(x,y)\, \mu(dy).
\end{align*}
In addition, since $c_1\leq p^r_{\delta}(x,y) \leq c_2$ for all $x,y \in \clD$, it also holds that $c_1 \leq \mu p_{\delta}(x) \leq c_2$ for all $x \in \clD$. Hence $\mu p_{\delta} \in \cD_{\delta}$. We set
\[
\cD:=\bigcup_{\delta>0}\cD_{\delta}.
\]
We simply denote $\rho dx \in \cD$ by $\rho$. Then it holds that
\[
\rho p_t (A) = \int_A p_t \rho (x)\, dx
\]
for $A \in \mathcal{B}(\clD)$. From this, we can also check that $\cD_{\delta}$ is $p_t$-invariant for any $\delta, t>0$. Indeed, for any $\rho \in \cD_{\delta}$ and $t>0$,
\[
c_1(\delta) \leq p_t \rho \leq c_2(\delta).
\]
Thus $\rho p_t \in \cD_{\delta}$.

We assume that we are given a function $J \colon \cD \to (-\infty, \infty]$. Then we define a function $J_{\delta} \colon \cP(\clD) \to (-\infty, \infty]$ by
\begin{align}
J_{\delta}(\mu)&:=J(\mu p_{\delta}).\label{defJdelta}
\end{align}
The structural optimization problem which we would like to solve is to find a minimizer of $J$ or $J_{\delta}$ if they exist.
\begin{lem}
If $J_{\delta}$ is proper lower semi continuous with respect to the metric $\cW_2$ on $\cP(\clD)$, then $J_{\delta}$ admits a minimizer on $\cP(\clD)$.
\end{lem}
\begin{proof}
Let $\mu_n$ be a minimizing sequence in $\cP(\clD)$. Then by the compactness of $(\cP(\clD),\cW_2)$, there exists a subsequence $n_k$ such that $\mu_n \to \mu$ in $\cW_2$. Thus
\[
J_{\delta}(\mu) \leq \lim_{k\to \infty}J_{\delta}(\mu_{n_k})= \inf_{\cP(\clD)}J_{\delta}.
\]
\end{proof}
The following theorem shows that the minimization problem for $J_{\delta}$ is actually an approximation of that for $J$.
\begin{lem}\label{approxLem}
Assume that for each $\delta>0$ and sequence $\rho_n$ in $\cD_{\delta}$ which converges to some $\rho \in \cD_{\delta}$ a.e., the function $J$ satisfies
\[
\liminf_{n\to \infty}J(\rho_n) \leq J(\rho).
\]
We also assume that for any $\delta>0$, $J_{\delta}$ admits at least one minimizer $\mu_{\delta}$. Then there exists a sequence $\delta_k$ with $\dis \lim_{k \to \infty}\delta_k=0$ such that
\begin{align}\label{minimumconv}
\lim_{k \to \infty}J_{\delta_k}(\mu_{\delta_k})=\inf_{\cD}J.
\end{align}
\end{lem}
\begin{proof}
Let $\rho \in \cD_{\bar{\delta}}$. Then for each $\delta > 0$,
\[
J_{\delta}(\mu_{\delta})\leq J_{\delta}(\rho)=J(p_{\delta}\rho).
\]
Since $\{p_t\}_{t>0}$ defines a strongly continuous semigroup on $L^2(D)$, there exists a sequence $\delta_k$ with $\dis \lim_{k\to \infty}\delta_k=0$ such that $\dis \lim_{k \to \infty}p_{\delta_k}\rho=\rho$ a.e. Since $p_{\delta_k} \rho$ is a sequence in $\cD_{\bar{\delta}}$, the assumption yields
\[
\liminf_{k\to \infty}J_{\delta_k}(\mu_{\delta_k})\leq \liminf_{k \to \infty}J(p_{\delta_k} \rho) \leq J(\rho).
\]
Since this holds for any $\bar{\delta}$ and $\rho \in \cD_{\bar{\delta}}$, we have
\[
\liminf_{k \to \infty}J_{\delta_k}(\mu_{\delta_k}) \leq \inf_{\cD}J.
\]
On the other hands, it holds that
\[
\inf_{\cD}J \leq J(\mu_{\delta}p_{\delta})=J_{\delta}(\mu_{\delta})
\]
for all $\delta > 0$. Thus we have \eqref{minimumconv}.
\end{proof}
We will provide some examples of $J$ and $J_{\delta}$ in the following subsections.

\subsection{Optimization for heat conduction}\label{minheat}
Let $D$ be a convex bounded domain in $\mathbb{R}^d$ with the partially smooth boundary. Let $\Gamma_0$ and $\Gamma_1$ be nonempty disjoint subsets in $\partial D$ and assume that $\Gamma_0$ has a positive $d-1$-dimensional Hausdorff measure. Let $\Gamma_2:=\partial D \backslash (\Gamma_0 \cup \Gamma_1)$. We denote the trace operator on $H^1(D)$ by $T \colon H^1(D)\to L^2(\partial D)$. Let $H^{-1}(D)$ be the dual space of $H^1_0(D)$. We set
\[
H^1_{\Gamma_0}(D):=\{ u \in H^1(D) \mid T u|_{\Gamma_0}=0 \}.
\]
We also set
\[
H^{\frac{1}{2}}(\Gamma_1):=T(H^1(D))|_{\Gamma_1}
\]
and denote the dual space of $H^{\frac{1}{2}}(\Gamma_1)$ by $H^{-\frac{1}{2}}(\Gamma_1)$. Let $\kappa \in C^1((0,\infty))$ be a bounded non-decreasing function with
\begin{align*}
\sup_{s\in (0,\infty)}(|\kappa'(s)|+|\kappa''(s)|)&<\infty,\\
\lim_{s \to 0} \kappa(s)&>0.
\end{align*}
For given $\rho \in \cD$, $f \in (H^1_{\Gamma_0})^*(D)$ and $g\in H^{-\frac{1}{2}}(\Gamma_1)$, there exists a unique $u=u_{\rho}$ satisfying
\begin{align}\label{gov1}
\int_{\Omega}\kappa(\rho) \nabla u \cdot \nabla \phi \, dx=\langle f, \phi \rangle_{(H^1_{\Gamma_0}(D))^*}+\langle g, T \phi|_{\Gamma_1} \rangle_{H^{-\frac{1}{2}}(\Gamma_1)}
\end{align}
for all $\phi \in H^1_{\Gamma_0}(D)$ by the Lax-Milgram theorem. If $f\in L^2(D)$, $g \in L^2(\Gamma_1)$ and $u=u_{\rho} \in H^2(D)$, the equation \eqref{gov1} corresponds to the PDE
\begin{align*}
\begin{cases}
-\mathrm{div}(\kappa (\rho) \nabla u )&=f\ \text{in}\ D,\\
u&=0\ \text{on}\ \Gamma_0,\\
\kappa(\rho) \partial_{\bn}u &= g\ \text{on}\ \Gamma_1,\\
\kappa(\rho) \partial_{\bn}u &= 0\ \text{on}\ \Gamma_2,
\end{cases}
\end{align*}
where $\bn$ is the outward unit normal vector field on $\partial D$ and $\partial_{\bn}$ is the outward normal derivative with the direction $\bn$.
We consider a functional $J \colon \cD \to $ defined by
\begin{align}\label{conductance}
J(\rho):=\langle f, u_{\rho} \rangle_{(H^1_{\Gamma_0})^*} + \langle g, Tu_{\rho} |_{\Gamma_1}\rangle_{H^{-\frac{1}{2}}(\Gamma_1)} - \frac{1}{2}\int_D \kappa(\rho(x))|\nabla u_{\rho}|^2(x)\, dx,
\end{align}
which is the sign-inverted potential energy in the thermal diffusion.
Since $u_{\rho}$ satisfies \eqref{gov1}, it holds that
\begin{align*}
J(\rho)=\frac{1}{2}\int_D \kappa(\rho(x))|\nabla u_{\rho}|^2(x)\, dx = \frac{1}{2}\left( \langle f, u_{\rho} \rangle_{(H^1_{\Gamma_0})^*} + \langle g, T u_{\rho} |_{\Gamma_1}\rangle_{H^{-\frac{1}{2}}(\Gamma_1)} \right).
\end{align*}
In particular, the optimization of $J$ is equivalent to the minimization of the weighted Dirichlet energy of $u_{\rho}$. Let $J_{\delta}$ be the functional defined by \eqref{defJdelta}.
\begin{thm}\label{wellposed1}
The functional $J_{\delta}$ is continuous on $\cP(\clD)$ with respect to the narrow topology. In addition, $J_{\delta}$ admits its minimum in $\cP(\clD)$.
\end{thm}
\begin{proof}
Let $\mu_n$ be a sequence in $\cD$ which converges to $\mu$ narrowly. We denote $u_n=u_{\mu_n p_{\delta}}$. Then for each $x \in D$,
\begin{align*}
\mu_n p^r_{\delta}(x) \to \mu p^r_{\delta}(x).
\end{align*}
In addition, $\{ \mu_n p_{\delta}\}_{n\in \mathbb{N}}$ is a sequence in $\cD_{\delta}$.
Therefore, the diffusion coefficient $\mu_np^r_{\delta}$ H-converges to $\mu p^r_{\delta}$ (see Lemma 1.2.22 in \cite{Allaire02} for details). In particular, we have
\[
u_n \to u\ \text{weakly in}\ H^1_{\Gamma_0}(D).
\]
Therefore, it holds that
\begin{align*}
J_{\delta}(\mu_n)&=J(\mu_n p_{\delta})\\
&=\langle f, u_n \rangle_{(H^1_{\Gamma_0})^*} + \langle g, T u_n|_{\Gamma_1} \rangle_{H^{-\frac{1}{2}}(\Gamma_1)} \\
&\to \langle f, u \rangle_{(H^1_{\Gamma_0})^*} + \langle g, T u|_{\Gamma_1} \rangle_{H^{-\frac{1}{2}}(\Gamma_1)}\ (\text{as}\ \to \infty) \\
&=J(\mu p_{\delta})\\
&=J_{\delta}(\mu).
\end{align*}
\end{proof}
\begin{prop}\label{approx1}
Let $\delta>0$. Let $\rho_n$ be a sequence in $\cD_{\delta}$ which converges to some $\rho \in \cD_{\delta}$ a.e. Then
\[
\lim_{n\to \infty}J(\rho_n) = J(\rho).
\]
\end{prop}
\begin{rem}
In particular, we can deduce that minimizers of $J_{\delta}$ approximates $\dis \inf_{\cD}J$ in some sense by \lref{approxLem}.
\end{rem}
\begin{proof}
Since the sequence $\rho_n$ is uniformly bounded from both above and below by constants, it also H-converges to $\rho$. Thus in the same way as the proof of \tref{wellposed1}, we have $\dis \lim_{n\to \infty}J(\rho_n)=J(\rho)$.
\end{proof}
\subsection{Optimization for mean compliance of linear elastic structures}\label{mincomp}
We continue to use the setting in the previous subsection for $D, \Gamma_0, \Gamma_1, \Gamma_2, \kappa$. Let $m \in \mathbb{N}$. In this subsection, we consider the vector valued Sobolev space $H^1_{\Gamma_0}(D;\mathbb{R}^m)$ and $H^{\frac{1}{2}}(\Gamma_1; \mathbb{R}^m)$ defined in the same way as in Subsection \ref{minheat}. For given $\rho \in \cD$, $f\in (H^1_{\Gamma_0}(D; \mathbb{R}^m))^*$ and $g \in H^{-\frac{1}{2}}(\Gamma_1; \mathbb{R}^m)$, let $u=u_{\rho} \in H_{\Gamma_0}^1(D; \mathbb{R}^m)$ solve the following weak form of the Lam\'e system
\begin{align*}
\int_{\Omega} \kappa(\rho)  \ep(\phi)\colon \sigma(u) \, dx &= \langle f,\phi \rangle_{(H^1_{\Gamma_0})^*}+\langle g, T \phi|_{\Gamma_1} \rangle_{H^{-\frac{1}{2}}(\Gamma_1)}\ \text{for all}\ \phi \in H^1_{\Gamma_0}(D; \mathbb{R}^m),
\end{align*}
where $\ep(u)$ and $\sigma(u)$ are defined by
\begin{align*}
\ep(u)&:=\frac{1}{2}\left( \nabla u + (\nabla u)^T \right),\\
\sigma(u)&:=2 \lambda_1 \ep(u) + \lambda_2 \mathrm{tr}(\ep(u)) I_d
\end{align*}
for $\lambda_1, \lambda_2 >0$, which are called Lam\'e's constants. Define a functional $J \colon \cD \to (-\infty,\infty]$ by
\begin{align}\label{MeanCompliance}
J(\rho)&:= \langle f,u_{\rho} \rangle_{(H^1_{\Gamma_0})^*} + \langle g, T u|_{\Gamma_1} \rangle_{H^{-\frac{1}{2}}(\Gamma_1)} - \frac{1}{2}\int_D \kappa(\rho) \ep(u_{\rho})\colon \sigma(u_{\rho})\, dx,
\end{align}
which is the sign-inverted potential energy. Then it holds that
\[
J(\rho)=\frac{1}{2}\int_D \kappa(\rho) \ep(u_{\rho})\colon \sigma(u_{\rho})\, dx = \frac{1}{2}\left( \langle f,u_{\rho} \rangle_{(H^1_{\Gamma_0})^*} + \langle g, T u|_{\Gamma_1} \rangle_{H^{-\frac{1}{2}}(\Gamma_1)} \right)
\]
by the state equation for $u_{\rho}$.
Then in the same way as \tref{wellposed1} and \pref{approx1}, we can show the following.
\begin{thm}\label{wellposed2}
\begin{enumerate}
\item The functional $J_{\delta}$ is narrowly continuous on $\cP(\clD)$. Thus consequently $J_{\delta}$ attains its minimum in $\cP(\clD)$.
\item Let $\mu_{\delta}$ be a minimizer of $J_{\delta}$. Then there exists a subsequence $\delta_k$ with $\dis \lim_{k \to \infty}\delta_k = 0$ such that
\[
\lim_{k \to \infty}J_{\delta_k}(\mu_{\delta_k})=\inf_{\cD}J.
\]
\end{enumerate}
\end{thm}
\section{Numerical algorithm}\label{numerics}
To begin with, we derive the Wasserstein gradient flow associated with the functional $J_{\delta}$ defined by \eqref{defJdelta} in order to obtain a candidate of a minimizer of $J_{\delta}$. In this section, we continue to use the notations in \sref{Formulation} but we additionally assume that the functional $J \colon \cD \to (-\infty,\infty]$ is given by the restriction of a functional $\cJ \colon L^{\infty}_{lbd}(D)\to (-\infty, \infty]$, where
\[
L^{\infty}_{lbd}(D)=\{ \rho \in L^{\infty}(D) \mid \text{There exists}\ c>0\ \text{such that}\ \rho \geq c\ \text{a.e.} \}.
\]
Let $\cM_+(\clD)$ be the set of positive finite Radon measures on $\clD$. Then for any $\delta>0$ and $\mu \in \cM_+(\clD)$,
\[
\mu p^r_{\delta}(x):=\int_{\clD}p^r_{\delta}(x,y)\, \mu(dy) \geq c^1\mu(\clD)>0.
\]
Therefore we can define a functional $\cJ_{\delta}\colon \cM_+(\clD) \to (-\infty, \infty]$ by
\[
\cJ_{\delta}(\mu):=\cJ(\mu p^r_{\delta}).
\]
\begin{prop}\label{Sensitivity}
Assume that $\cJ$ is G\^ateaux differentiable at each $\rho \in \cD$ in the direction $\psi \in C^{\infty}(D)$ and there exists $S[\rho] \in L^1(D)$ such that
\[
\lim_{\tau \to 0}\frac{1}{\tau}(\cJ(\rho + \tau \psi)-\cJ(\rho))=\int_D S(x)\psi(x)\, dx.
\]
\begin{enumerate}
\item The functional $\cJ_{\delta}$ is G\^ateaux differentiable on $\cP(\clD)$ in the direction of $C^{\infty}(D)$ and
\begin{align*}
\lim_{\tau \to 0}\frac{1}{\tau}(\cJ_{\delta}(\rho + \tau \psi)-\cJ_{\delta}(\rho))&=\int_D S_{\delta}(x)\psi(x)\, dx,\\
S_{\delta}[\mu]&=p_{\delta}S[\mu p_{\delta}].
\end{align*}
In particular, $S_{\delta}[\mu] \in C^{\infty}(D) \cap C(\overline{D})$ and it satisfies the Neumann boundary condition on $\partial D$.
\item For each $\mu \in \cP(\clD)$,
\begin{align}\label{GradFormula}
\nabla^{\cW}J_{\delta}[\mu]=-\mathrm{div}\mu \nabla S_{\delta}[\mu].
\end{align}
\end{enumerate}
\end{prop}
\begin{proof}
\begin{enumerate}
\item Let $\psi \in C^{\infty}(D)$. Then
\begin{align*}
\frac{1}{\tau}(\cJ_{\delta}(\mu + \tau \psi)-\cJ_{\delta}(\mu)) &= \frac{1}{\tau}(\cJ(\mu p_{\delta} + \tau p_{\delta}\psi)-\cJ(\mu p_{\delta}))\\
&\to \int_D S[\mu p_{\delta}](x)p_{\delta}\psi(x)\, dx\ (\text{as}\ \tau \to 0)\\
&=\int_Dp_{\delta}\left( S[p_{\delta}\rho] \right)(x)\psi(x)\, dx.
\end{align*}
Thus we have $S_{\delta}[\mu]=p_{\delta}S[\mu p_{\delta}]$.
\item Let $\mu_t$ be a curve in $\cP(\clD)$ with its initial value $\mu_0=\mu$ satisfying the continuity equation with the velocity vector field $\nabla \phi$ with $\phi \in C^{\infty}(\clD)$ near $t=0$. Then in the same way as Proposition 3 in \cite{OkaYama25}, which dealt with the case of the whole $\mathbb{R}^d$, we have
\begin{align*}
\left( \frac{d}{d t}\right)_{t=0}J_{\delta}(\mu_t)=\langle \nabla \phi, \nabla S_{\delta}[\mu] \rangle_{L^2(\mu)}.
\end{align*}
Since $S_{\delta}[\mu]$ is smooth and satisfies $\partial_{\bn}S_{\delta}[\mu]=0$ on $\partial D$, we have
\[
\langle \nabla \phi, \nabla S_{\delta}[\mu] \rangle_{L^2(\mu)}=\langle -\mathrm{div}(\mu \nabla S_{\delta}[\mu]),\phi \rangle_{\dot{H}^{-1}_{\mu}(\clD)}.
\]
Thus we obtain \eqref{GradFormula}.
\end{enumerate}
\end{proof}
\begin{ex}\label{diffheat}
Assume the setting of Subsection \ref{minheat}. In this case the functional $J$ defined by \eqref{conductance} can be realized as the restriction of the functional $\cJ$ on $L^{\infty}_{lbd}(D)$ defined straightforward just by regarding the density $\rho$ as an element in $L^{\infty}_{lbd}(D)$. It is well-known that one can derive the G\^ateaux differential of $\cJ$ as
\begin{align*}
\lim_{\tau \to 0}\frac{1}{\tau}(\cJ(\rho + \tau \psi)-\cJ(\rho))&=\int_D S \psi,
\end{align*}
where
\begin{align}\label{heatSensitivity}
S(x)&=\kappa'(\rho (x))|\nabla u_{\rho}|^2(x)
\end{align}
for $\rho \in L^{\infty}_{lbd}(D)$ and $\psi \in C^{\infty}(D)$.
\end{ex}
\begin{ex}\label{diffcomp}
In the setting of Subsection \ref{mincomp}, we can also define the functional $\cJ$ on $L^{\infty}_{lbd}(D)$ as an extension of $J$ defined by \eqref{MeanCompliance}. The G\^ateaux differential of $\cJ$ is given by
\begin{align*}
&\lim_{\tau \to 0}\frac{1}{\tau}(\cJ(\rho + \tau \psi)-\cJ(\rho))=\int_D S \psi,\\
&S(x)=\kappa'(\rho (x))\ep(u_{\rho}) : \sigma(u_{\rho})(x)
\end{align*}
for $\rho \in L^{\infty}_{lbd}(D)$ and $\psi \in C^{\infty}(D)$.
\end{ex}

In view of \pref{Sensitivity}, the Wasserstein gradient flow of the functional $J_{\delta}$ is
\begin{align}\label{gradientflow}
\partial_t\mu_t=\mathrm{div}\left( \mu_t \nabla S_{\delta}[\mu_t] \right).
\end{align}
Then
\begin{align*}
\left(\frac{d}{d t}\right)_{t=0} J_{\delta}(\mu_t)=-\int_D |\nabla S_{\delta}[\mu](x)|^2\, \mu(dx) \leq 0.
\end{align*}

However, the numerical calculation by the update following \eqref{gradientflow} tends to be unstable. Thus we modify the gradient flow by filtering the sensitivity as follows. Fix $\ep>0$. For $\eta>0$, let $S^{\eta}[\mu]$ be the unique solution to
\begin{align}\label{filteredS}
\begin{cases}
-\eta \mathrm{div}(\mu p^r_{\ep} \nabla S^{\eta}[\mu]) + S^{\eta}[\mu]=S_{\delta}[\mu]\ &\text{in}\ D,\\
\partial_{\bn}S^{\eta}[\mu]=0\ &\text{on}\ \partial D
\end{cases}
\end{align}
and consider the continuity equation
\begin{align}\label{filteredGF}
\partial_t \mu^{\eta}_t = \mathrm{div}(\mu_t^{\eta}\nabla S^{\eta}[\mu_t^{\eta}]).
\end{align}
We prepare some notations of norms on spaces of smooth functions. Assume that $D$ has the $C^k$-boundary. Let $C^k(\clD)$ be a space of functions in $C^k(\mathbb{R}^d)$ restricted to $\clD$. We denote $C(\clD)=C^0(\clD)$. For $u \in C^k(\clD)$, we set
\[
\|u\|_{C^k(\clD)}=\sum_{i = (i_1,\dots,i_m)\in \mathrm{MI}_k} \sup_{x\in \clD}|\partial_i u (x)|,
\]
where
\begin{align*}
\mathrm{MI}_k&=\{i=(i_1,\dots,i_m)\mid m\in \mathbb{N},\ i_1,\dots,i_m \in \mathbb{N},\ |i| \leq k \},\\
|i|&=i_1+\cdots+i_m,\\
\partial_i&=\partial_{i_1}\cdots \partial_{i_m},\ i=(i_1,\dots,i_m)\in \mathrm{MI}_k.
\end{align*}
For $\alpha \in (0,1)$ and $u\in C(\clD)$, we set
\begin{align*}
[u]_{C^{0,\alpha}(\clD)}:=\sup_{x,y\in \clD}\frac{|u(x)-u(y)|}{|x-y|^{\alpha}}.
\end{align*}
We also define the H\"{o}lder norm of $u\in C^k(\clD)$ by
\[
\|u\|_{C^{k,\alpha}(\clD)}:=\|u\|_{C^k(\clD)}+\sup_{i\in \mathrm{MI}_k,\ |i|=k}[\partial_i u]_{C^{0,\alpha}(\clD)}.
\]

\begin{lem}\label{ApproxLem}
Assume that $D$ has the $C^3$-boundary.
Let $\cT\colon \cP(\clD) \ni \mu \mapsto \cT_{\mu} \in L^1(D)$ and assume that there exists $A_1>0$ such that
\[
\|\cT_{\mu} \|_{L^1(D)} \leq A_1 
\]
for all $\mu \in \cP(\clD)$. We fix $\delta>0$ and define $\cS \colon \cP(\clD) \ni \mu \mapsto \cS_{\mu} \in C^{\infty}(D)$ by
\[
\cS_{\mu}:=p_{\delta}\cT_{\mu}.
\]
Then there exists $A_2, A_3>0$ depending only on $\delta$ and $D$ such that for all $\mu, \nu \in \cP(\clD)$ and $x,y \in \clD$,
\begin{align}
\int_D |\nabla \cS_{\mu}(x)- \nabla \cS_{\nu}(x)|^2\, dx &\leq A_2 \| \cT_{\mu}-\cT_{\nu} \|^2_{L^1(D)},\label{muLip}\\
|\nabla \cS_{\mu}(x)- \nabla \cS_{\mu}(y)| &\leq A_3 |y-x|.\label{xLip}
\end{align}
In addition, if we further fix $\ep>0$ and define $\cS^{\eta}_{\mu} \in C^{\infty}(D)$ by
\begin{align*}
\begin{cases}
-\eta \mathrm{div}(\mu p_{\ep} \, \nabla \cS^{\eta}_{\mu}) + \cS^{\eta}_{\mu} = \cS_{\mu}\ &\text{in}\ D,\\
\partial_{\bn} \cS^{\eta}_{\mu} = 0\ &\text{on}\ \partial D.
\end{cases}
\end{align*}
Then for any fixed $\gamma \in (0,\frac{1}{2})$, there exists $A_4=A_4(\gamma, \ep, \delta, D)>0$ such that for all $\mu \in \cP(\clD)$ and $\eta >0$,
\begin{align}
\| \nabla \cS^{\eta}_{\mu} - \nabla \cS_{\mu} \|_{C(\clD)} \leq A_4 \eta^{\frac{1}{2}-\gamma}. \label{etaApprox}
\end{align}
\end{lem}
\begin{proof}
By simple calculation, we have
\begin{align*}
|\nabla \cS_{\mu}(x)-\nabla \cS_{\nu}(x)| &\leq \int_D \nabla_x p^r_{\delta}(x,y)|(T_{\mu}(y)-T_{\nu}(y))|\, dy\\
&\leq \|p^r_{\delta} \|_{C^1(\clD \times \clD)} \|T_{\mu}-T_{\nu} \|_{L^1(D)}.
\end{align*}
Thus we obtain \eqref{muLip}. In a similar way, we have
\begin{align*}
|\nabla S_{\mu}(x)- \nabla S_{\mu}(y)| &\leq \int_D | \nabla_x p^r_{\delta}(x,z)-\nabla_y p^r_{\delta}(y,z) | T_{\mu}(z)\, dz\\
&\leq A_1 \|p_{\delta}^r\|_{C^2(\clD \times \clD)} |x-y|.
\end{align*}
Thus we have \eqref{xLip}. Next, we set $w:=\cS^{\eta}_{\mu}-\cS_{\mu}$. Then $w$ satisfies
\begin{align*}
\begin{cases}
-\eta \mathrm{div}(\mu p_{\ep} \, \nabla w) + w = -\eta \mathrm{div}(\mu p_{\ep} \nabla \cS_{\mu})\ &\text{in}\ D,\\
\partial_{\bn} w = 0\ &\text{on}\ \partial D.
\end{cases}
\end{align*}
In order to obtain \eqref{etaApprox}, we employ the Schauder estimate for $w$. We fix $\gamma \in (0,\frac{1}{2})$ and set $\alpha=\frac{4\gamma}{1+2\gamma} \in (0,1)$. We let $F=\mathrm{div}(\mu p_{\ep} \nabla \cS_{\mu})$. Hereafter, we denote constants depending only on $D$, $\gamma$, $\delta$ and $\ep$ by the same notation $A$ to simplify the argument. By Lemma 2.2 of \cite{Kono25}, we have
\begin{align}\label{maximal}
\| w \|_{C(\clD)} \leq \eta \|F\|_{C(\clD)}.
\end{align}
On the other hand, by Theorem 6.30 in \cite{GT}, it holds that
\begin{align}\label{Schauder}
\| w \|_{C^{2,\alpha}(\clD)} \leq A (\|w\|_{C(\clD)}+\|F-\eta^{-1}w\|_{C^{0,\alpha}(\clD)})
\end{align}
for some $A$. By the interpolation inequality regarding H\"{o}lder spaces (e.g. Lemma 6.35 in \cite{GT}), for each $\theta, \theta'>0$, 
\begin{align}
[ w ]_{C^{0,\alpha}(\clD)} &\leq A\theta^{-\frac{\alpha}{1-\alpha}}\|w\|_{C(\clD)}+\theta \|w\|_{C^1(\clD)},\\
\|w\|_{C^{1}(\clD)} &\leq A (\theta')^{-1}\|w\|_{C(\clD)}+\theta' \|w\|_{C^2(\clD)}.
\end{align}
Combining them with \eqref{maximal} and \eqref{Schauder}, we have
\begin{align*}
\|w\|_{C^{1}(\clD)} &\leq A (\theta')^{-1}\| w \|_{C(\clD)}+A\theta'(\|w\|_{C(\clD)}+\|F-\eta^{-1}w\|_{C^{0,\alpha}(\clD)})\\
&\leq A((\theta')^{-1} + \theta')\|w\|_{C(\clD)}+A\theta'\|F\|_{C^{0,\alpha}(\clD)}\\
&\h +\frac{A\theta'}{\eta}((\theta^{-\frac{\alpha}{1-\alpha}}+1)\|w\|_{C(\clD)}+\theta \|w\|_{C^1(\clD)}).
\end{align*}
Thus we have
\begin{align}
(1-\frac{A\theta \theta'}{\eta})\|w\|_{C^1(\clD)} \leq A(\theta'+\frac{\eta}{\theta'}+\eta \theta'+\theta'\theta^{-\frac{\alpha}{1-\alpha}})\|F\|_{C^{0,\alpha}(\clD)}.
\end{align}
Choosing $\theta=\eta^{\frac{1-\alpha}{2-\alpha}}$ and $\theta'=\frac{1}{2A}\eta^{\frac{1}{2-\alpha}}$, we have
\begin{align}
\|w\|_{C^1(\clD)} \leq A \eta^{\frac{1-\alpha}{2-\alpha}}\|F\|_{C^{0,\alpha}(\clD)}=A\eta^{\frac{1}{2}-\gamma}\|F\|_{C^{0,\alpha}(\clD)}.
\end{align}
Note that $\|F\|_{C^{0,\alpha}(\clD)}$ can be bounded by a constant which depends only on $\gamma$, $\delta$, $\ep$ and $D$. Indeed, we have
\begin{align*}
\|\cS_{\mu}\|_{C^{1,\alpha}(\clD)} &\leq A_1 \|p^r_{\delta}\|_{C^{1,\alpha}(\clD \times \clD)},\\
\|\cS_{\mu}\|_{C^{2,\alpha}(\clD)} &\leq A_1 \|p^r_{\delta}\|_{C^{2,\alpha}(\clD \times \clD)}.
\end{align*}
Therefore, the right-hand side of the following inequality
\begin{align*}
\|F\|_{C^{0,\alpha}(\clD)}&= \|\nabla \mu p_{\ep} \cdot \nabla \cS_{\mu}+\mu p_{\ep} \Delta \cS_{\mu}\|_{C^{0,\alpha}(\clD)}\\
&\leq \|\nabla \mu p_{\ep} \|_{C^{0,\alpha}(\clD)} \|\nabla \cS_{\mu}\|_{C^{0,\alpha}(\clD)}+\| \mu p_{\ep}\|_{C^{0,\alpha}(\clD)}\|\Delta \cS_{\mu}\|_{C^{0,\alpha}(\clD)}\\
&\leq \|\mu p_{\ep}\|_{C^{1,\alpha}(\clD)}\|\cS_{\mu}\|_{C^{1,\alpha}(\clD)} + \| \mu p_{\ep} \|_{C^{0,\alpha}(\clD)}\| \cS_{\mu} \|_{C^{2,\alpha}(\clD)}
\end{align*}
is bounded from above by a constant depending only on $\gamma, \delta, \ep$ and $D$.
Thus we obtain \eqref{etaApprox}.
\end{proof}
\begin{thm}\label{ErrorEstimate}
Assume that $D$ has the $C^3$-boundary. Let $J$ be a functional on $\cD$ defined by \eqref{conductance} in Subsection \ref{minheat} or \eqref{MeanCompliance} in Subsection \ref{mincomp}. We additionally assume that $f \in L^2(D)$ and $g\in L^2(\Gamma_1)$. Let $\{\mu_t^{\eta}\}_{t\in [0,T]}$ be a curve on $\cP(\barD)$ satisfying \eqref{filteredGF} with the initial state $\mu^{\eta}_0=\bar{\mu}$ which is independent of $\eta$. We set $\mu_t:=\mu^0_t$. Then for any $\gamma \in (0,\frac{1}{2})$, there exists a constant $A=A(\delta, \ep, \gamma,D,f,g,\kappa)$ such that for each $t \in [0,T]$,
\begin{align}
\cW_2(\mu_t,\mu^{\eta}_t) \leq A \eta^{\frac{1}{2}-\gamma} t e^{At}.\label{order}
\end{align}
In particular,
\[
\sup_{t\in [0,T]}\cW_2(\mu_t,\mu^{\eta}_t) = O(\eta^{\frac{1}{2}-\gamma})\ \text{as}\ \eta \to 0.
\]
\end{thm}
\begin{proof}
Let $\pi_t^{\eta}$ be an optimal coupling of $\mu_t^{\eta}$ and $\mu_t$ with respect to the $2$-Wasserstein distance. Then we have the differential formula (Theorem 8.4.7, Remark 8.4.8, and Proposition 8.5.4 in \cite{AGS21})
\begin{align*}
\frac{d}{dt}\cW_2(\mu^{\eta}_t,\mu_t)^2&=\int_{\clD\times \clD}\langle x-y, \nabla S^{\eta}[\mu^{\eta}_t](x)-\nabla S[\mu_t](y) \rangle \, d\pi^{\eta}_t(dxdy)\\
&\leq \cW_2(\mu^{\eta}_t,\mu_t) \left( \int_{\clD \times \clD} |\nabla S^{\eta}[\mu^{\eta}_t](x)-\nabla S[\mu_t](y)|^2 \, \pi^{\eta}_t(dxdy) \right)^{\frac{1}{2}}.
\end{align*}
We set
\begin{align*}
I_1&:= \int_{\clD \times \clD} |\nabla S^{\eta}[\mu^{\eta}_t](x)-\nabla S[\mu^{\eta}_t](x)|^2 \, \pi^{\eta}_t(dxdy),\\
I_2&:= \int_{\clD \times \clD} |\nabla S[\mu^{\eta}_t](x)-\nabla S[\mu_t](x)|^2 \, \pi^{\eta}_t(dxdy),\\
I_3&:= \int_{\clD \times \clD} \nabla S[\mu_t](x)-\nabla S[\mu_t](y)|^2 \, \pi^{\eta}_t(dxdy).
\end{align*}
We apply \lref{ApproxLem} by setting $T_{\mu}:=S[\mu p_{\delta}]$. First we show the uniform $L^1$-bound of $T_{\mu}$. We assume the setting of Subsection \ref{minheat} and $J$ is defined by \eqref{conductance}. We set
\begin{align*}
k_{\delta}&:=\kappa(c_1),\\
K_{\delta}&:=\kappa(c_2),
\end{align*}
where $c_1$ and $c_2$ are as given by \eqref{infHK} and \eqref{supHK}, respectively. We also denote $\kappa(\mu p_{\delta})$ and $\kappa'(\mu p_{\delta})$ by $\kappa_{\mu}$ and $\kappa'_{\mu}$, respectively. Then we have
\begin{align*}
k_{\delta}\int_D |\nabla u_{\mu p_{\delta}}|^2 \, dz &\leq \int_D \kappa_{\mu} |\nabla u_{\mu p_{\delta}}|^2\, dz\\
&=\int_D fu_{\mu p_{\delta}} + \int_{\Gamma_1}g u_{\mu p_{\delta}}\\
&\leq \| u \|_{L^2(D)} \| f \|_{L^2(D)} + \| g \|_{L^2(\Gamma_1)} \| u \|_{H^1(D)}\\
&\leq C(\|f \|_{L^2(D)}+\| g \|_{L^2(\Gamma_1)}) \| \nabla u_{\mu p_{\delta}} \|_{L^2(D)},
\end{align*}
where we have used the continuity the trace operator in the second inequality and the Poincar\'e inequality in the third inequality, which is valid due to $|\Gamma_0| > 0$, and $C$ is a constant associated with it. Thus we obtain the uniform bound
\begin{align}
\|\nabla u_{\mu p_{\delta}}\|_{L^2(D)}\leq \frac{C}{k_{\delta}}(\|f\|_{L^2(D)}+\|g\|_{L^2(\Gamma_1)}).\label{DuL2}
\end{align}
Therefore, we have
\begin{align*}
\| S[\mu p_{\delta}]\|_{L^1(D)} &\leq \| \kappa'\|_{L^{\infty}([c_1,c_2])} \| \nabla u_{\mu p_{\delta}}\|^2_{L^2(D)}\\
&\leq \frac{\| \kappa'\|_{L^{\infty}([c_1,c_2])}C^2}{k_{\delta}^2} (\|f \|_{L^2(D)}+\| g \|_{L^2(D)})^2.
\end{align*}
Thus by applying \lref{ApproxLem}, for each fixed $\gamma \in (0,\frac{1}{2})$, there exists $A>0$ such that
\begin{align*}
I_1 &= \int_{\clD} |\nabla S^{\eta}_{\delta}[\mu_t^{\eta}]-\nabla S_{\delta}[\mu_t]|^2 \, d\mu^{\eta}_t \leq A \eta^{1-2\gamma},\\
I_2 &\leq A \| S[\mu_t p_{\delta}]-S[\mu^{\eta}_tp_{\delta}] \|^2_{L^1(D)},\\
I_3 &\leq A \int_{\clD\times \clD}|x-y|^2\, \pi^{\eta}_t(dxdy)=A \cW_2(\mu_t,\mu^{\eta}_t)^2.
\end{align*}
Hereafter, we denote constants depending only on $D$, $\gamma$, $\delta$, $\ep$, $f$, $g$ and $\kappa$ by the same notation $A$ to simplify the argument.
Next, we estimate the upper bound of $I_2$. It holds that
\begin{align*}
\| S[\mu_t p_{\delta}]-S[\mu^{\eta}p_{\delta}] \|_{L^1(D)}&\leq \int_D|\kappa'_{\mu}-\kappa'_{\nu}||\nabla u_{\mu p_{\delta}}|^2 + |\kappa'_{\nu}|||\nabla u_{\mu p_{\delta}}|^2-|\nabla u_{\nu p_{\delta}}|^2|.
\end{align*}
We denote the first and second terms by $H_1$ and $H_2$, respectively. Note that it holds that
\begin{align}\label{kappaWass}
\|\kappa_{\mu}-\kappa_{\nu} \|_{L^{\infty}(D)} & \leq \|\kappa'\|_{L^{\infty}([c_1,c_2])}\|\mu p_{\delta}-\nu p_{\delta} \|_{C(\clD)} \nonumber \\
&\leq \|\kappa'\|_{L^{\infty}([c_1,c_2])} \|p^r_{\delta} \|_{C^1(\clD \times \clD)}\cW_1(\mu,\nu),
\end{align}
where we used the Kantorovich duality in the second inequality. In the same way, we have
\begin{align}\label{kappa'Wass}
\|\kappa'_{\mu}-\kappa'_{\nu} \|_{L^{\infty}(D)} \leq \|\kappa''\|_{L^{\infty}([c_1,c_2])}\|p^r_{\delta} \|_{C^1(\clD \times \clD)}\cW_1(\mu,\nu).
\end{align}
Thus we have
\begin{align*}
H_1 \leq A \cW_1(\mu,\nu)
\end{align*}
for some $A$ by \eqref{DuL2}. Next, if we let $v=u_{\mu p_{\delta}}-u_{\nu p_{\delta}}$, then $v$ satisfies
\begin{align*}
\begin{cases}
\mathrm{div}\kappa_{\mu}\nabla v =\mathrm{div}(\kappa_{\mu}-\kappa_{\nu})\nabla u_{\mu p_{\delta}}\ &\text{in}\ D,\\
v=0\ &\text{on}\ \Gamma_0,\\
\partial_n v=(\kappa_{\mu}^{-1}-\kappa_{\nu}^{-1})g\ &\text{on}\ \Gamma_1,\\
\partial_n v=0\ &\text{on}\ \Gamma_2.
\end{cases}
\end{align*}
Therefore, for any $\phi \in C^{\infty}(D)$, it holds that
\[
\int_D\kappa_{\mu}\nabla v \cdot \nabla \phi \, dx=-\int_D (\kappa_{\mu}-\kappa_{\nu})\nabla u_{\mu p_{\delta}}\cdot \nabla \phi \, dx + \int_{\Gamma_1}\phi\{(\kappa_{\mu}-\kappa_{\nu})\kappa_{\nu}^{-1}+\kappa_{\mu}(\kappa_{\mu}^{-1}+\kappa_{\nu}^{-1}) \}g.
\]
Taking $\phi=v$, we have
\begin{align*}
\|\nabla v\|_{L^2(D)}^2 &\leq k_{\delta}^{-1} \int_D \kappa_{\mu}|\nabla v|^2\\
&\leq \| \kappa_{\mu}-\kappa_{\nu}\|_{L^{\infty}(D)}\|\nabla u_{\mu p_{\delta}}\|_{L^2(D)} \|\nabla v\|_{L^2(D)}\\
&\h +\left( k_{\delta}^{-1}+\frac{K_{\delta}}{k_{\delta}^2} \right)\|v\|_{H^1(D)}\|g\|_{L^2(\Gamma_1)} \|\kappa_{\mu}-\kappa_{\nu}\|_{L^{\infty}(D)}.
\end{align*}
By the Poincar\'e inequality and \eqref{DuL2}, we have
\[
\|\nabla v\|_{L^2(D)} \leq A \|\kappa_{\mu}-\kappa_{\nu}\|_{L^{\infty}(D)}
\]
for some $A$. Combining with \eqref{kappaWass}, we have
\[
H_2 \leq A \cW_1(\mu,\nu).
\]
Therefore,
\[
\| S[p_{\delta} \mu]-S[\nu p_{\delta}] \|_{L^1(D)} \leq A\cW_1(\mu,\nu).
\]
Since $\cW_1(\mu,\nu)\leq \cW_2(\mu,\nu)$ by H\"{o}lder's inequality, we conclude that there exists a constant $A$ such that
\[
\frac{d}{dt}\cW_2(\mu_t,\mu^{\eta}_t) \leq A(\cW_2(\mu_t,\mu^{\eta}_t)+\eta^{\frac{1}{2}-\gamma})
\]
for a.e. $t\in [0,T]$. By integrating both sides and applying Gronwall's inequality, we obtain \eqref{order}. The same proof works for the setting of Subsection \ref{mincomp} by combining Korn's inequality with the Poincar\'e inequality.
\end{proof}
We adopt the filtered gradient flow \eqref{filteredGF} for the update of the distribution of mass. In practice, we fix the time step $\tau >0$ and solve the discretized equation
\begin{align}\label{update}
\int_D \psi(x) \frac{\rho_{i+1}(x)-\rho_i(x)}{\tau} \, dx - \int_D\rho_i(x) \nabla \psi(x)\cdot \nabla S^{\eta}_i(x)\, dx = 0,\ \psi \in H^1(D)
\end{align}
by the finite element method. We show the procedure of the computation in the following Algorithm in the case of the thermal diffusion. We assume that $f \in L^2(D)$ and $g \in L^2(\Gamma_1)$. For fixed $a>0$ and $p>0$, we set
\[
\kappa (s)=\left(\frac{1}{1-e^{-as}} \right)^p.
\]

\begin{algorithm}
\caption{Optimization of the density function}
\begin{enumerate}
\item Let $i=0$. Set $a, \delta, \ep, \eta, \tau>0$.
\item Solve the following equation for $\tilde \rho_i$:
\begin{align*}
\begin{cases}
\tilde{\rho_i}-\rho_i &= \delta \Delta \tilde \rho_i\ \text{in}\ D,\\
\partial_{\bn}\tilde{\rho_i}&=0\ \text{on}\ \partial{D}.
\end{cases}
\end{align*}
\item Solve the governing equation for $u_i$:
\begin{align*}
\begin{cases}
-\mathrm{div}(\kappa(\tilde{\rho_i}) \nabla u_i) &=f\ \text{in}\ D,\\
u_i&=0\ \text{on}\ \Gamma_0,\\
\partial_{\bn}u_i&=g\ \text{on}\ \Gamma_1,\\
\partial_{\bn}u_i&=0\ \text{on}\ \Gamma_2.
\end{cases}
\end{align*}
\item Set $S[\tilde{\rho_i}]$ to \eqref{heatSensitivity} and solve the equation for $S_{\delta}[\rho_i]$:
\begin{align*}
\begin{cases}
S_{\delta}[\rho_i]-S[\tilde{\rho_i}] &= \delta \Delta S_{\delta}[\rho_i]\ \text{in}\ D,\\
\partial_{\bn}S_{\delta}[\rho_i]&=0\ \text{on}\ \partial{D}.
\end{cases}
\end{align*}
\item Solve the following equation for $\bar \rho_i$:
\begin{align*}
\begin{cases}
\bar{\rho_i}-\rho_i &= \ep \Delta \bar \rho_i\ \text{in}\ D,\\
\partial_{\bn}\bar{\rho_i}&=0\ \text{on}\ \partial{D}.
\end{cases}
\end{align*}
\item Solve the following equation for $S^{\eta}_{\delta}[\rho_i]$.
\begin{align*}
\begin{cases}
S^{\eta}_{\delta}[\rho_i]-S_{\delta}[\rho_i] &= \eta \mathrm{div}(\bar \rho_i \nabla S^{\eta}_{\delta}[\rho_i])\ \text{in}\ D,\\
\partial_{\bn}S^{\eta}_{\delta}[\rho_i]&=0\ \text{on}\ \partial{D}.
\end{cases}
\end{align*}
\item Update the density by solving
\begin{align*}
\begin{cases}
\rho_{i+1}-\rho_i&=\tau \mathrm{div}(\rho_i \nabla S^{\eta}_{\delta}), \\
\partial_{\bn}\rho_{i+1}&=0.
\end{cases}
\end{align*}
\item Let $i=i+1$ and go back to (2).
\end{enumerate}
\end{algorithm}

\clearpage
\section{Numerical results}\label{results}
Based on Algorithm 1, we perform the numerical calculation in the setting of Subsection \ref{minheat} and \ref{mincomp}. In the algorithm of the optimization, we have parameters $a, \delta, \ep, \eta, \tau >0$. We construct candidates of optimal mass distribution with $\delta = 1.0 \times 10^{-2}, 1.0 \times 10^{-3}, 1.0 \times 10^{-4}$ and $\eta = 1.0 \times 10^{-2}, 1.0 \times 10^{-3}, 1.0 \times 10^{-4}$ and exhibit the figures of obtained results. Following the custom of the SIMP method, we take $p=3$. We fix $\ep=1.0 \times 10^{-7}$.

\subsection{Heat conduction}\label{heatresult}
First we exhibit the result of numerical calculation in the setting of Subsection \ref{minheat}.
We set $a=1.3$. The design domain is set to $D=[0,1]\times [0,1]$. The boundary for the zero-Dirichlet condition is set to $\Gamma_0=\{0\} \times [0.5-0.06,0.5+0.06]$. Then we set $\Gamma_1 = \emptyset$ and $\Gamma_2=\partial D \backslash \Gamma_0$, i.e. we do not impose the non-homogeneous Neumann boundary condition. We set the constant heat source $f=0.5$. We set the initial distribution to the uniform distribution, i.e. $\mu_0=dx$. The results are exhibited in Figure \ref{figmass1} below. Since we do not restrict the value of the density to the interval $[0,1]$, the range of those values varies for each $(\delta,\eta)$. However, we can still obtain the final shape of the material since the values of the densities are separated into $0$ and the higher values due to the penalization by $\kappa$ and $p=3$.
\begin{figure}[H]
   \centering
   \begin{minipage}[b]{0.3\textwidth}
      \includegraphics[width=5cm]{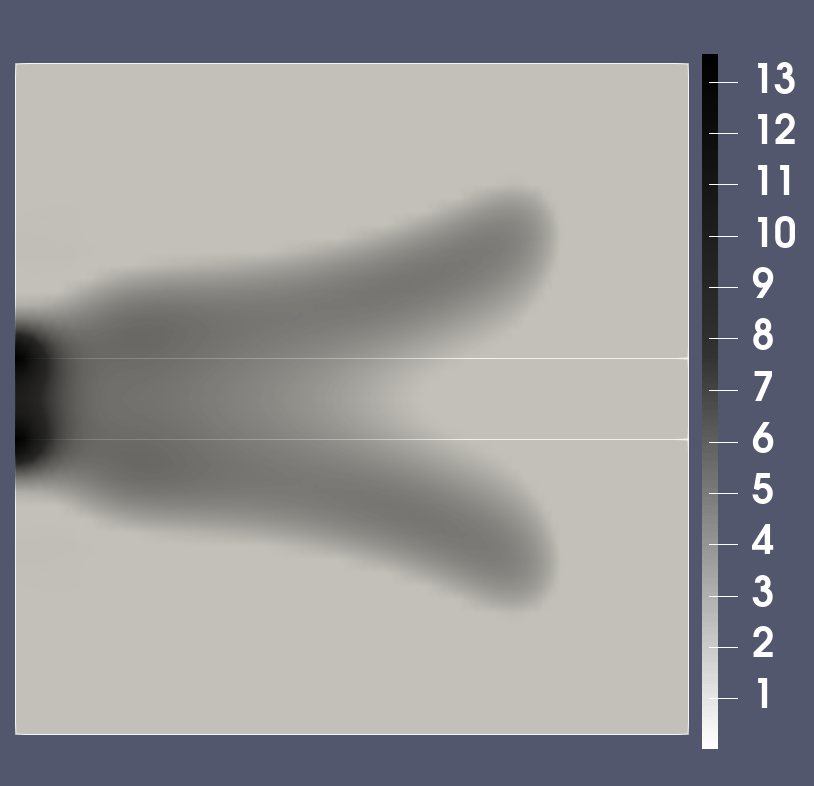}
      \caption*{$(\delta,\eta)=(10^{-2},10^{-2})$}
   \end{minipage}
   \h
   \begin{minipage}[b]{0.3\textwidth}
      \includegraphics[width=5cm]{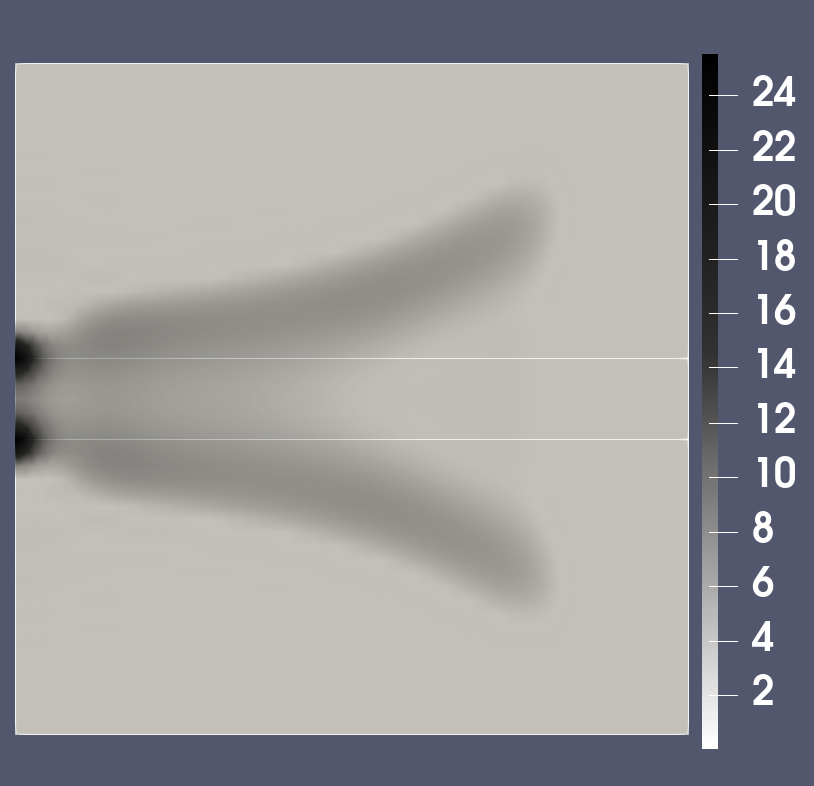}
      \caption*{$(\delta,\eta)=(10^{-2},10^{-3})$}
   \end{minipage}
   \h
   \begin{minipage}[b]{0.3\textwidth}
      \includegraphics[width=5cm]{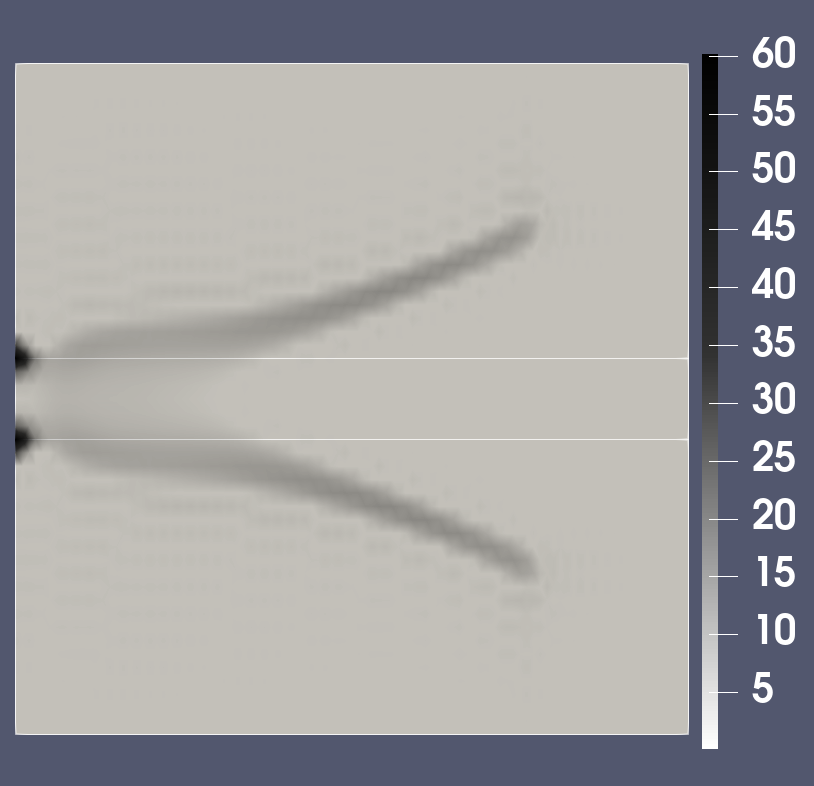}
      \caption*{$(\delta,\eta)=(10^{-2},10^{-4})$}
   \end{minipage}
\end{figure}

\begin{figure}[H]
   \centering
   \begin{minipage}[b]{0.3\textwidth}
      \includegraphics[width=5cm]{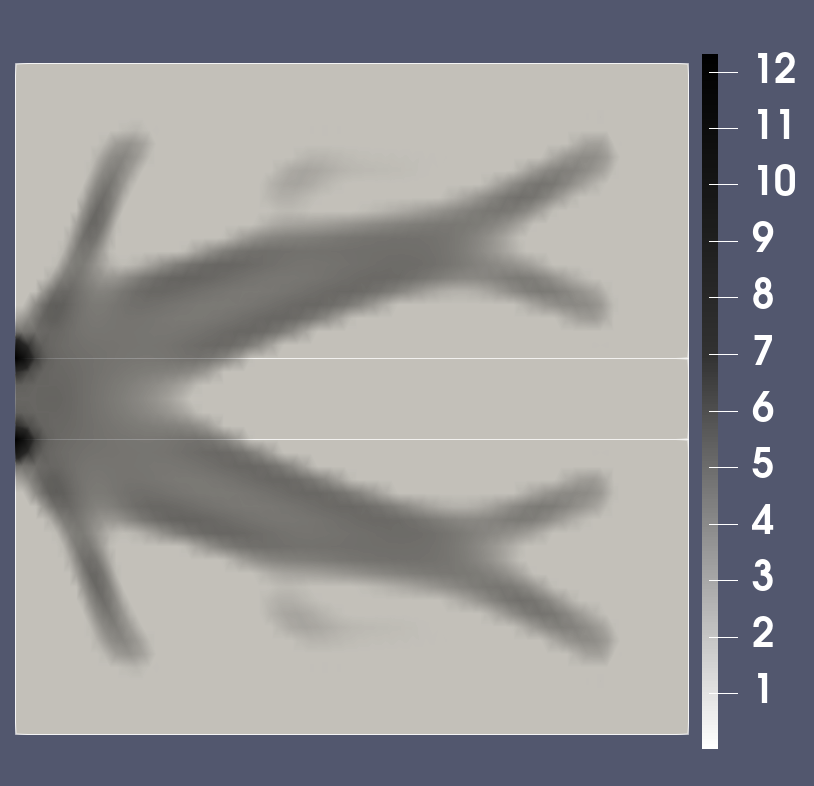}
      \caption*{$(\delta,\eta)=(10^{-3},10^{-2})$}
   \end{minipage}
   \h
   \begin{minipage}[b]{0.3\textwidth}
      \includegraphics[width=5cm]{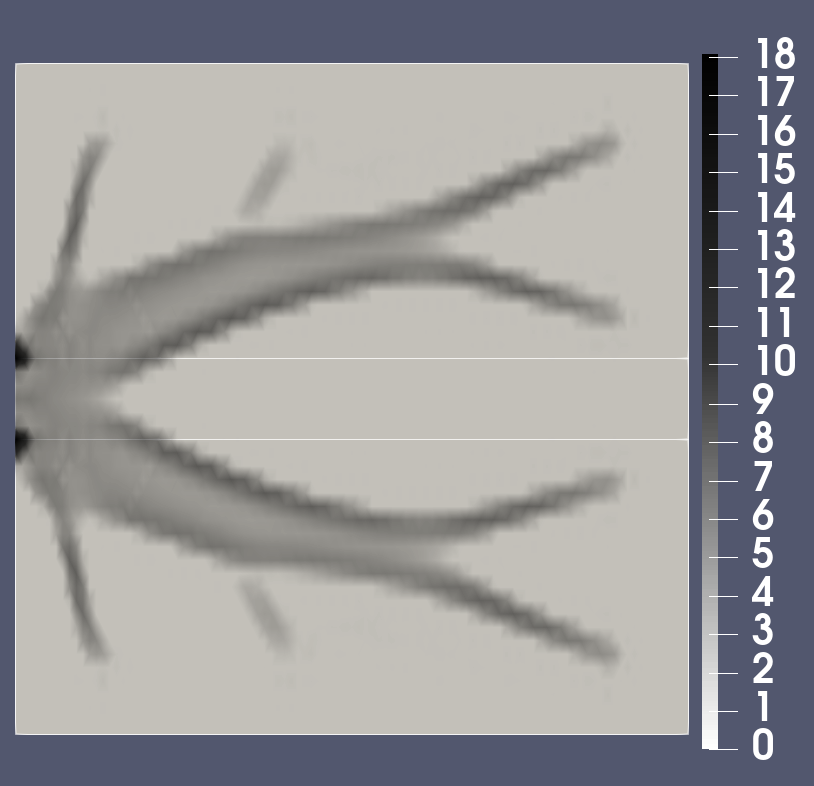}
      \caption*{$(\delta,\eta)=(10^{-3},10^{-3})$}
   \end{minipage}
   \h
   \begin{minipage}[b]{0.3\textwidth}
      \includegraphics[width=5cm]{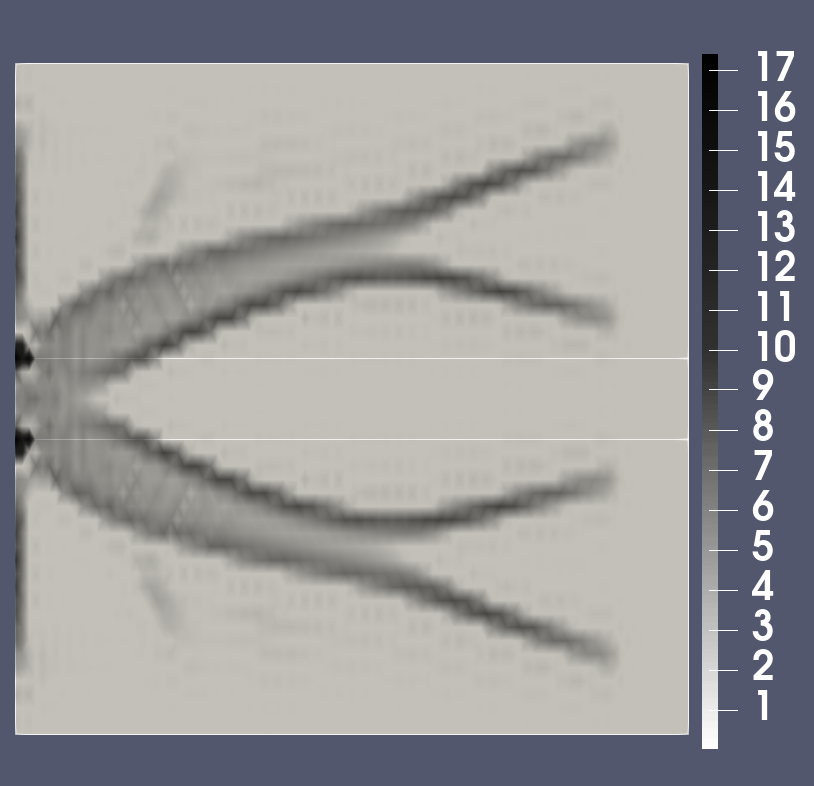}
      \caption*{$(\delta,\eta)=(10^{-3},10^{-4})$}
   \end{minipage}
\end{figure}

\begin{figure}[H]
   \centering
   \begin{minipage}[b]{0.3\textwidth}
      \includegraphics[width=5cm]{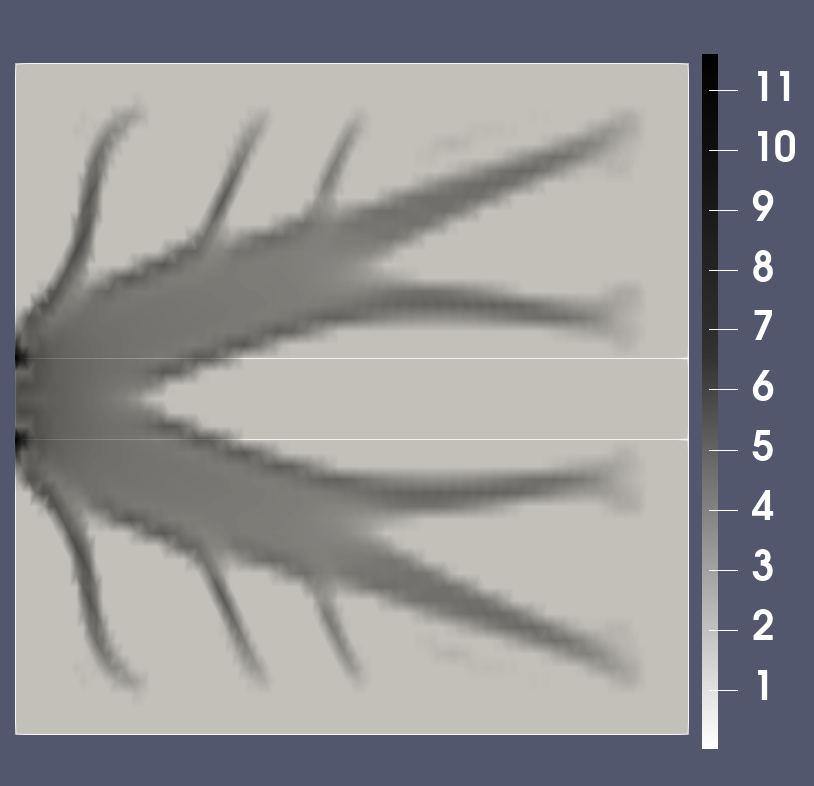}
      \caption*{$(\delta,\eta)=(10^{-4},10^{-2})$}
   \end{minipage}
   \h
   \begin{minipage}[b]{0.3\textwidth}
      \includegraphics[width=5cm]{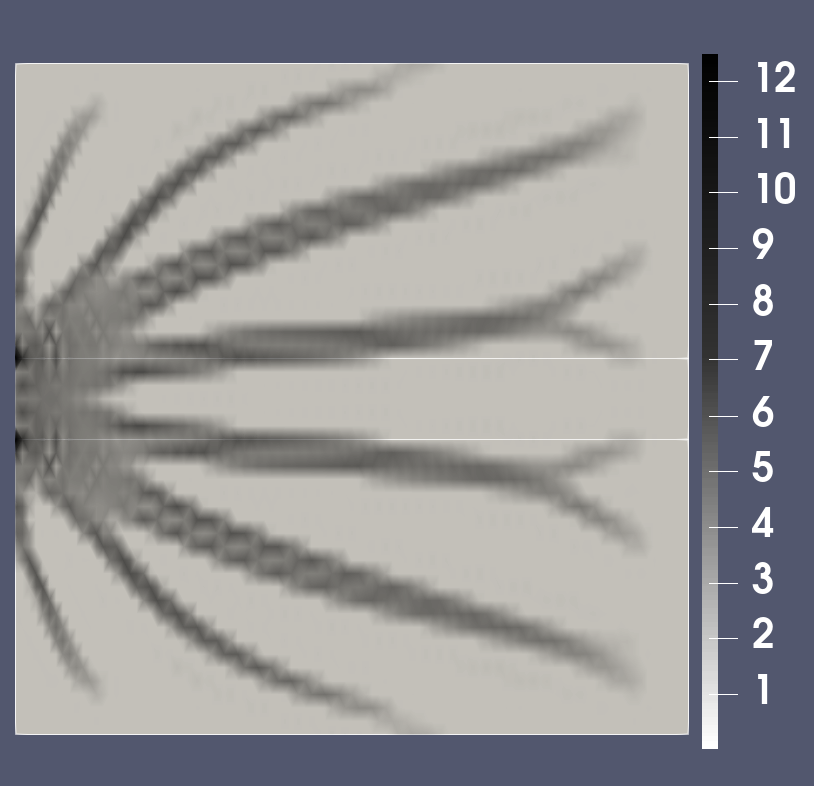}
      \caption*{$(\delta,\eta)=(10^{-4},10^{-3})$}
   \end{minipage}
   \h
   \begin{minipage}[b]{0.3\textwidth}
      \includegraphics[width=5cm]{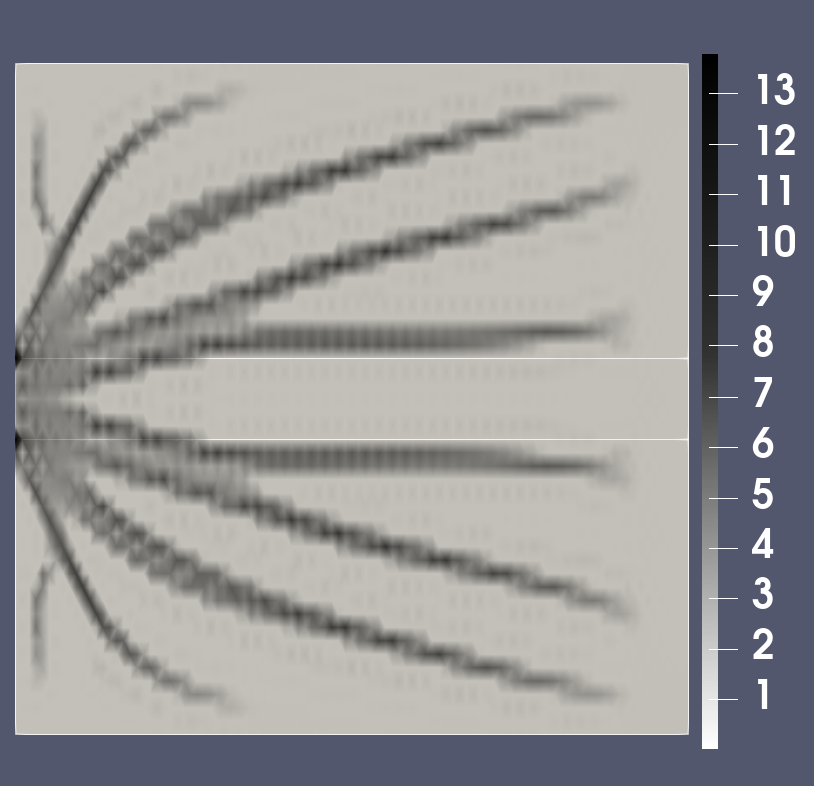}
      \caption*{$(\delta,\eta)=(10^{-4},10^{-4})$}
   \end{minipage}
\caption{Mass distribution}\label{figmass1}
\end{figure}
Figure \ref{figobj1} below describes how the objective function decreases for each pair $(\delta,\eta)$. Note that the time step $\tau$ is set differently for each $(\delta, \eta)$ since we do not care about the rate of convergence in this paper. In our numerical calculation, we set $\tau=1.0 \times 10^{-3}$ for $(\delta,\eta)=(10^{-2},10^{-2}), (10^{-2},10^{-3}), (10^{-2},10^{-4}), (10^{-3},10^{-2}), (10^{-3},10^{-3}), (10^{-3},10^{-4}), (10^{-4},10^{-2})$ and $\tau=3.0 \times 10^{-4}$ for $(\delta,\eta)=(10^{-4},10^{-3}),(10^{-4},10^{-4})$.
\begin{figure}[H]
   \centering
   \begin{minipage}[b]{0.3\textwidth}
      \includegraphics[width=5cm]{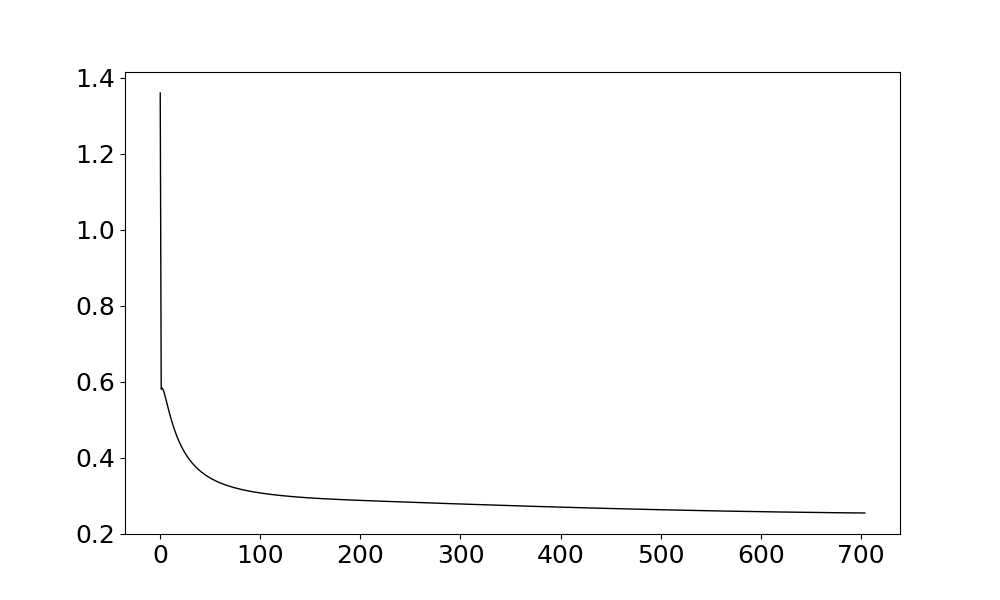}
      \caption*{$(\delta,\eta)=(10^{-2},10^{-2})$}
   \end{minipage}
   \h
   \begin{minipage}[b]{0.3\textwidth}
      \includegraphics[width=5cm]{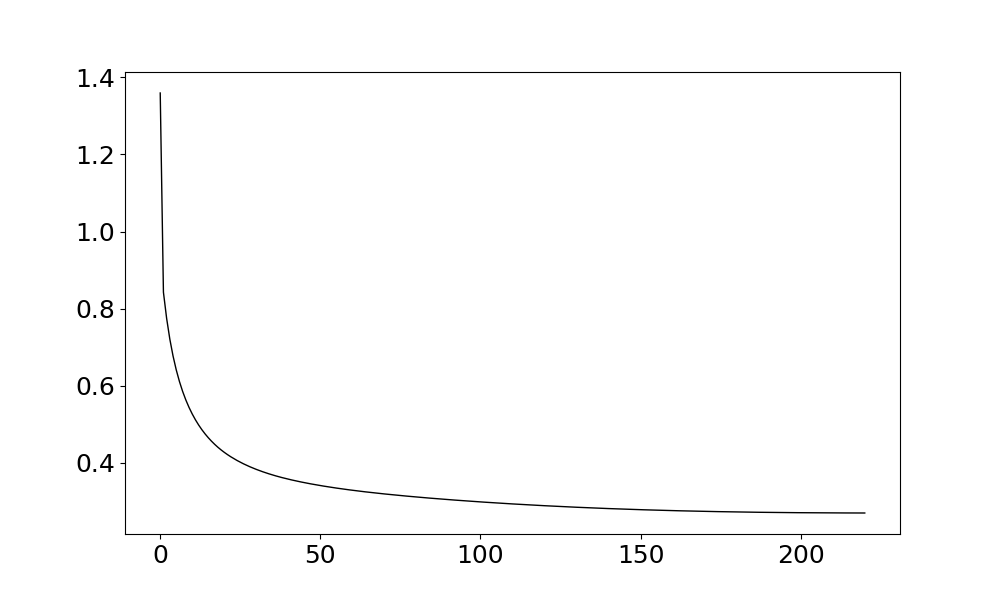}
      \caption*{$(\delta,\eta)=(10^{-2},10^{-3})$}
   \end{minipage}
   \h
   \begin{minipage}[b]{0.3\textwidth}
      \includegraphics[width=5cm]{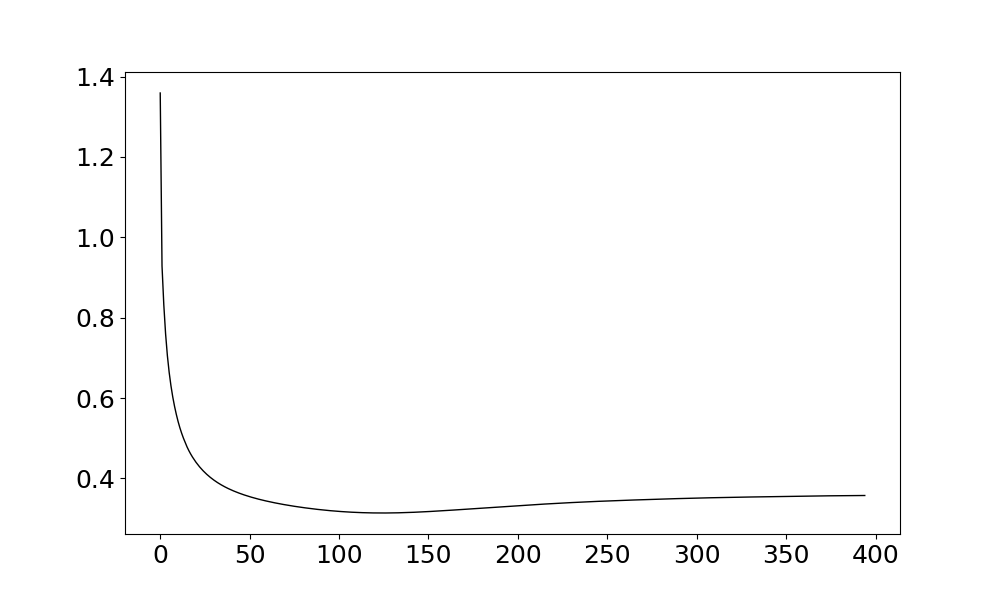}
      \caption*{$(\delta,\eta)=(10^{-2},10^{-4})$}
   \end{minipage}
\end{figure}

\begin{figure}[H]
   \centering
   \begin{minipage}[b]{0.3\textwidth}
      \includegraphics[width=5cm]{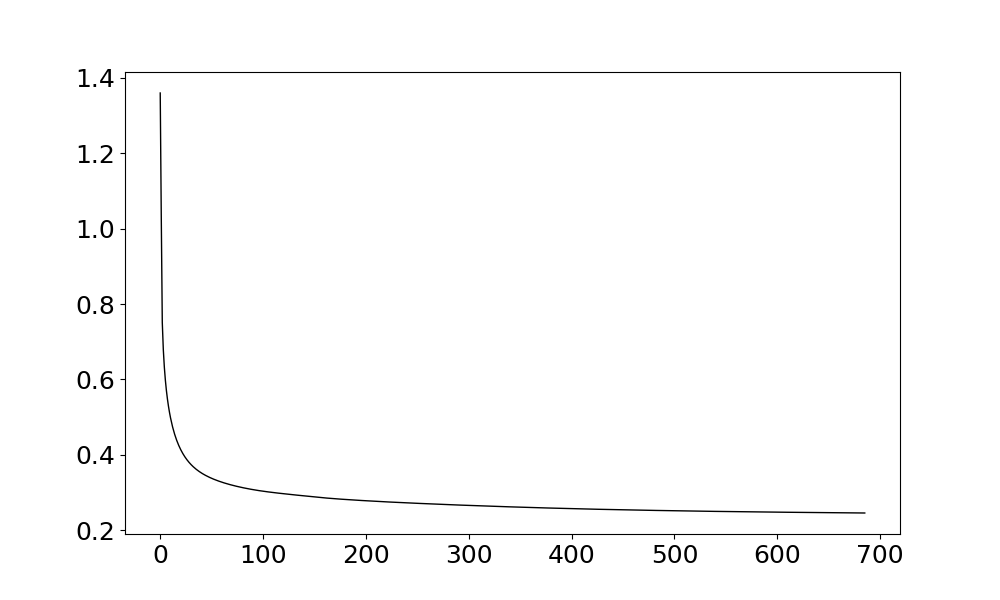}
      \caption*{$(\delta,\eta)=(10^{-3},10^{-2})$}
   \end{minipage}
   \h
   \begin{minipage}[b]{0.3\textwidth}
      \includegraphics[width=5cm]{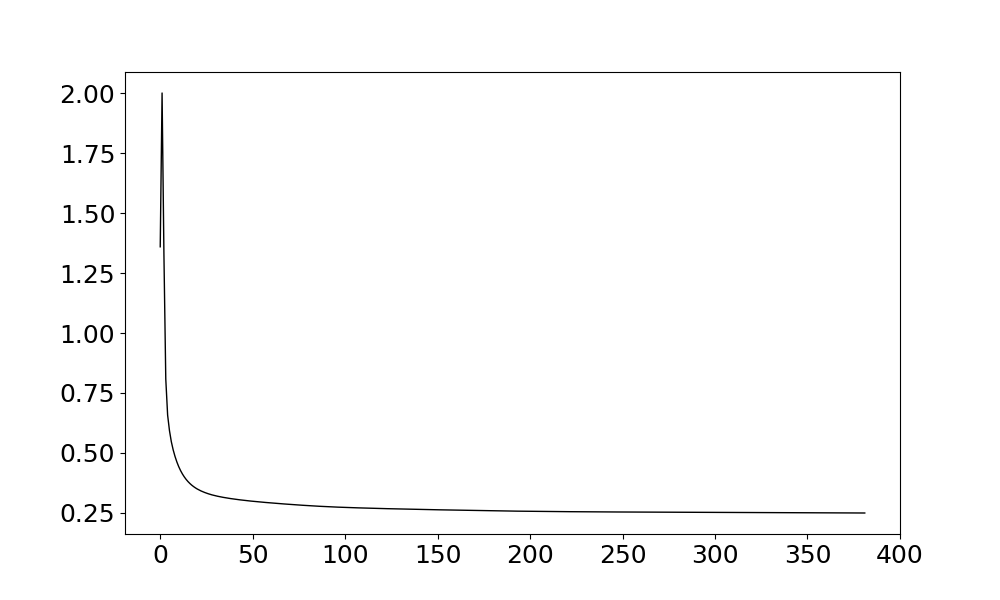}
      \caption*{$(\delta,\eta)=(10^{-3},10^{-3})$}
   \end{minipage}
   \h
   \begin{minipage}[b]{0.3\textwidth}
      \includegraphics[width=5cm]{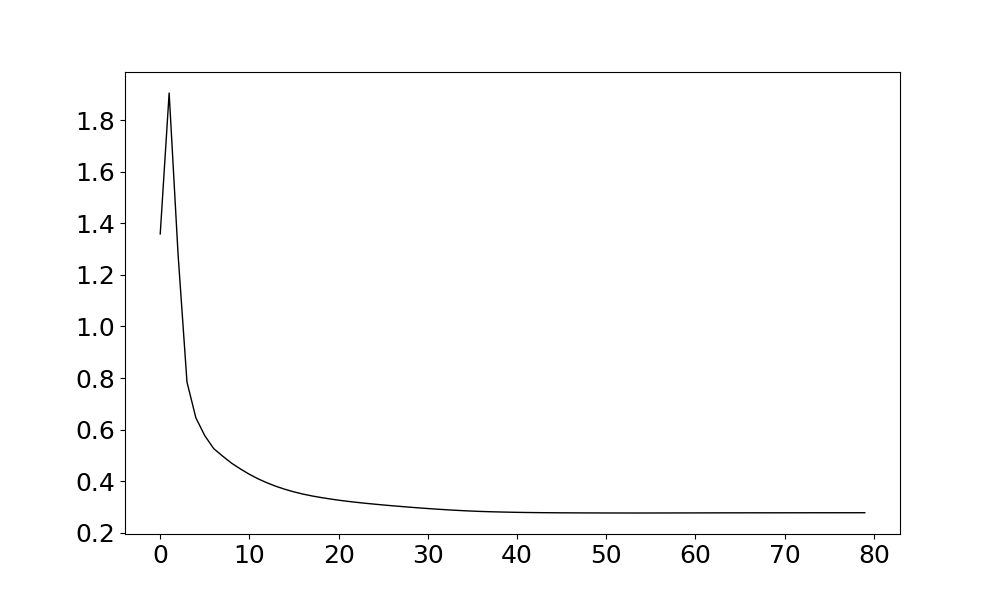}
      \caption*{$(\delta,\eta)=(10^{-3},10^{-4})$}
   \end{minipage}
\end{figure}

\begin{figure}[H]
   \centering
   \begin{minipage}[b]{0.3\textwidth}
      \includegraphics[width=5cm]{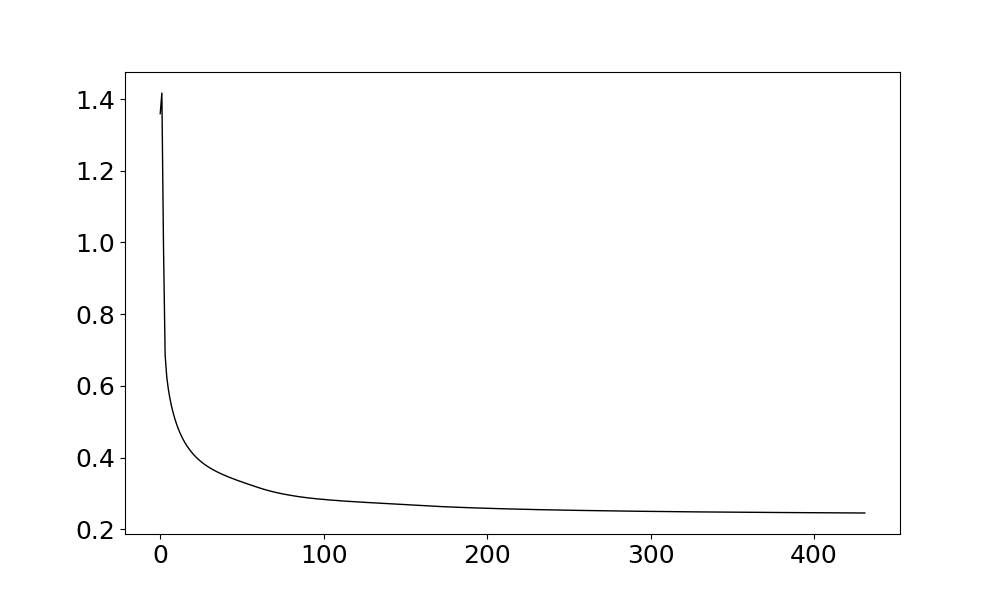}
      \caption*{$(\delta,\eta)=(10^{-4},10^{-2})$}
   \end{minipage}
   \h
   \begin{minipage}[b]{0.3\textwidth}
      \includegraphics[width=5cm]{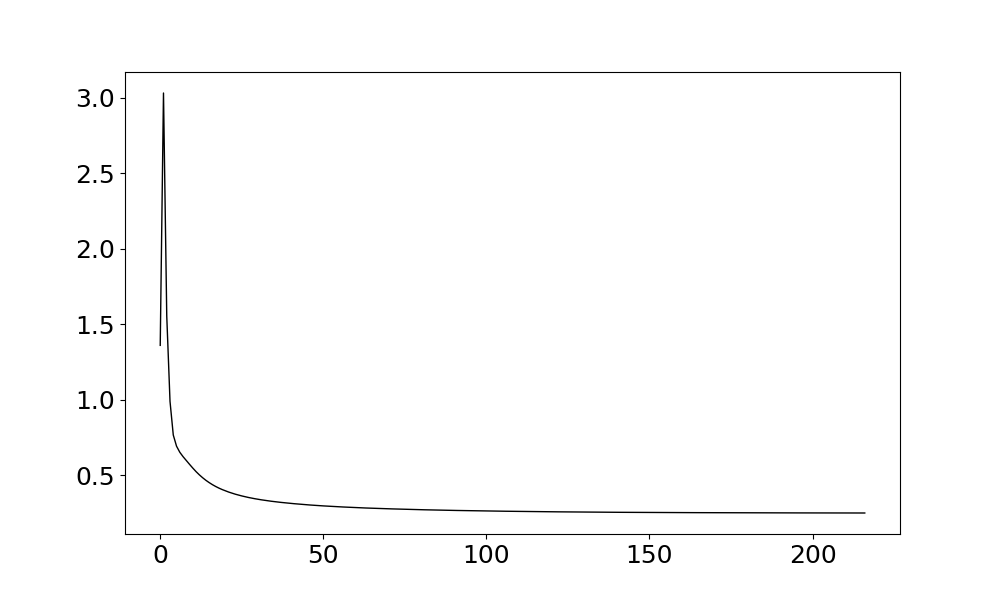}
      \caption*{$(\delta,\eta)=(10^{-4},10^{-3})$}
   \end{minipage}
   \h
   \begin{minipage}[b]{0.3\textwidth}
      \includegraphics[width=5cm]{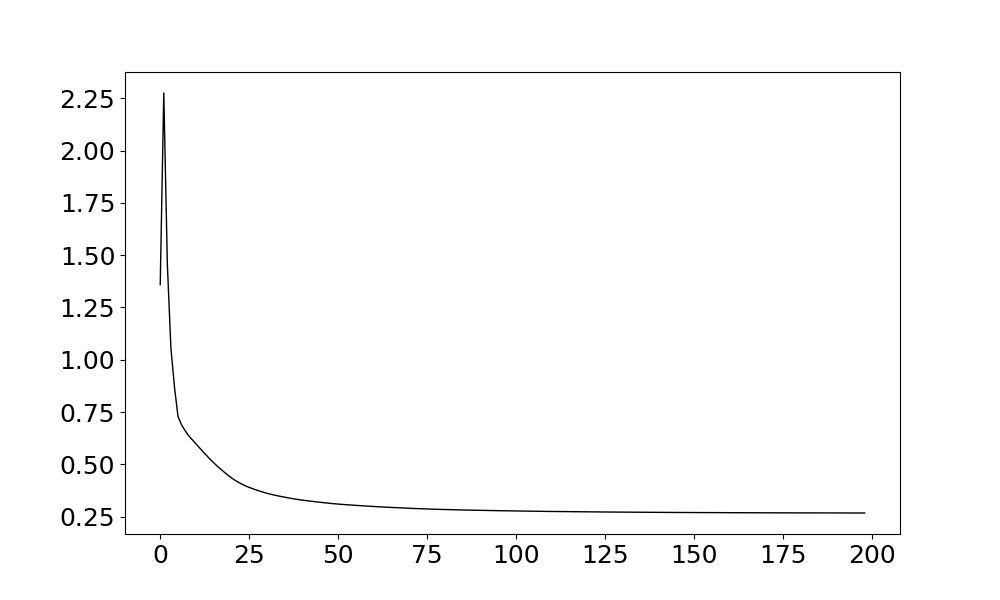}
      \caption*{$(\delta,\eta)=(10^{-4},10^{-4})$}
   \end{minipage}
\caption{Convergence history of the objective function}\label{figobj1}
\end{figure}
We also exhibit Figure \ref{figerr1} which describes the history of the total mass. The lateral axis indicates the step number and the vertical axis indicates the logarithm of the relative error with respect to the initial mass, namely, $\mathrm{log}\frac{\mu^{\eta}_t(\clD)}{\mu_0(\clD)}$. From these results, we observe that we can let the objective function decrease while keeping the total mass of the material.
\begin{figure}[H]
   \centering
   \begin{minipage}[b]{0.3\textwidth}
      \includegraphics[width=5cm]{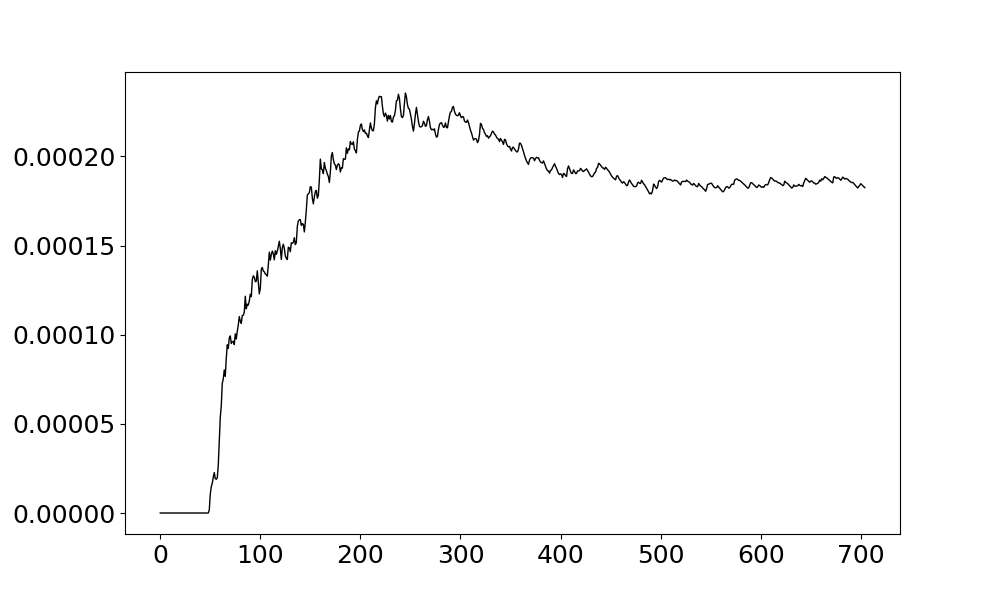}
      \caption*{$(\delta,\eta)=(10^{-2},10^{-2})$}
   \end{minipage}
   \h
   \begin{minipage}[b]{0.3\textwidth}
      \includegraphics[width=5cm]{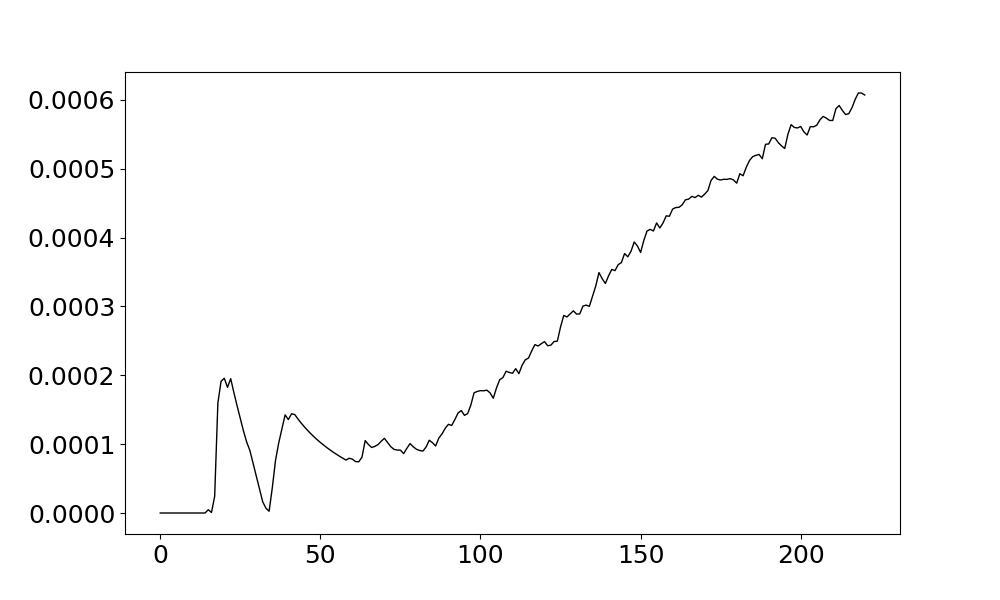}
      \caption*{$(\delta,\eta)=(10^{-2},10^{-3})$}
   \end{minipage}
   \h
   \begin{minipage}[b]{0.3\textwidth}
      \includegraphics[width=5cm]{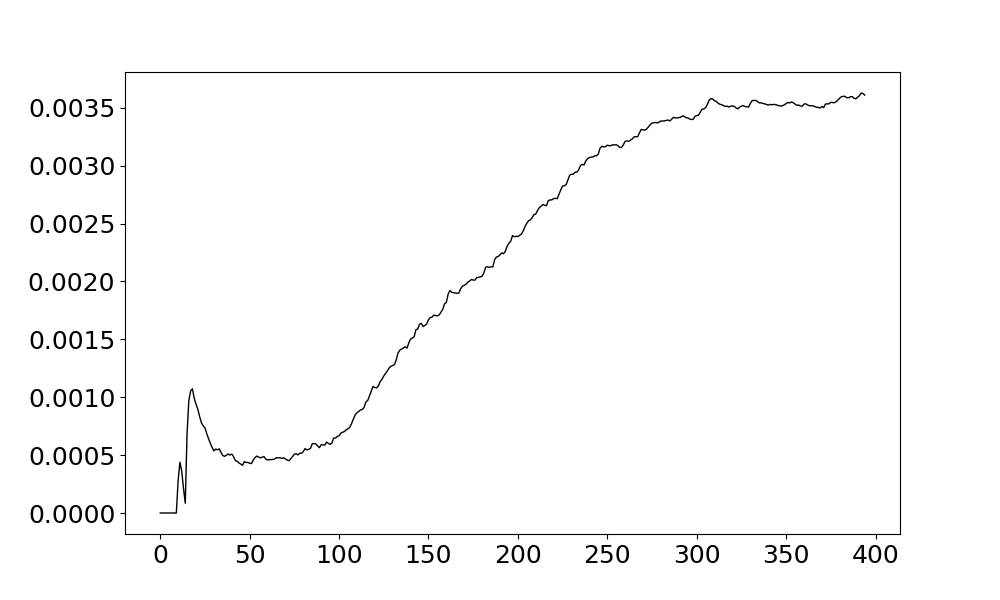}
      \caption*{$(\delta,\eta)=(10^{-2},10^{-4})$}
   \end{minipage}
\end{figure}

\begin{figure}[H]
   \centering
   \begin{minipage}[b]{0.3\textwidth}
      \includegraphics[width=5cm]{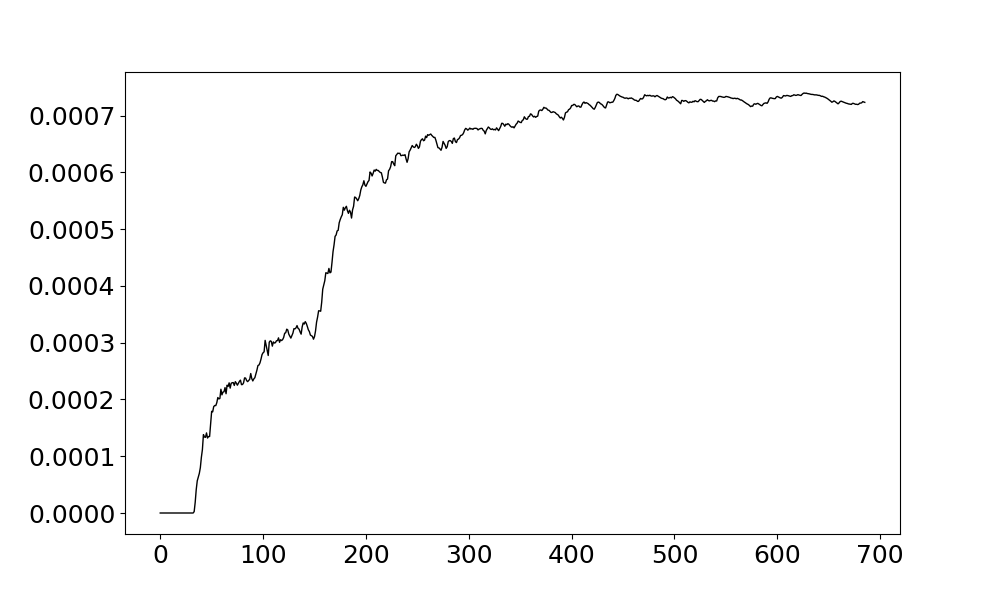}
      \caption*{$(\delta,\eta)=(10^{-3},10^{-2})$}
   \end{minipage}
   \h
   \begin{minipage}[b]{0.3\textwidth}
      \includegraphics[width=5cm]{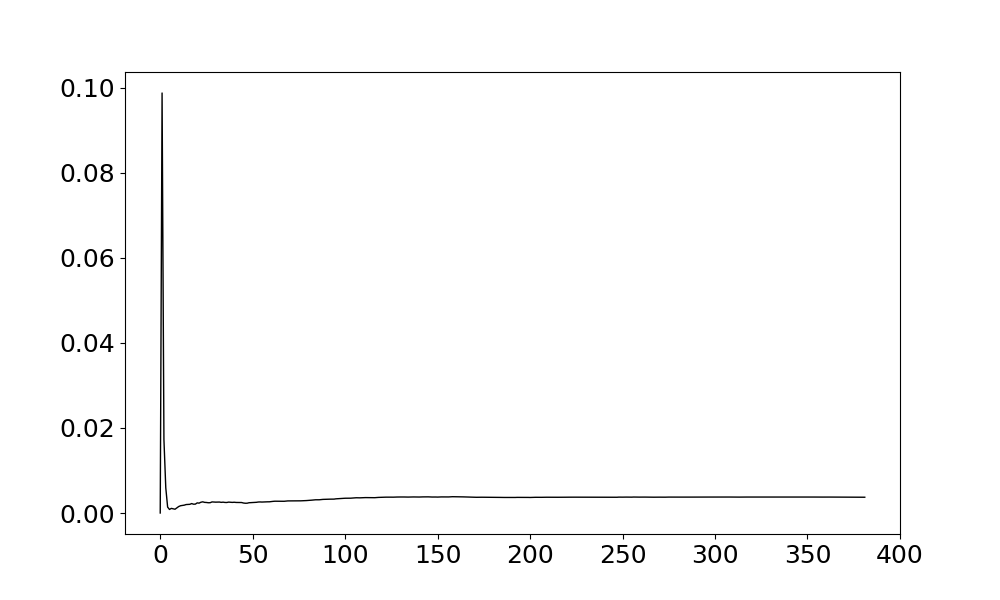}
      \caption*{$(\delta,\eta)=(10^{-3},10^{-3})$}
   \end{minipage}
   \h
   \begin{minipage}[b]{0.3\textwidth}
      \includegraphics[width=5cm]{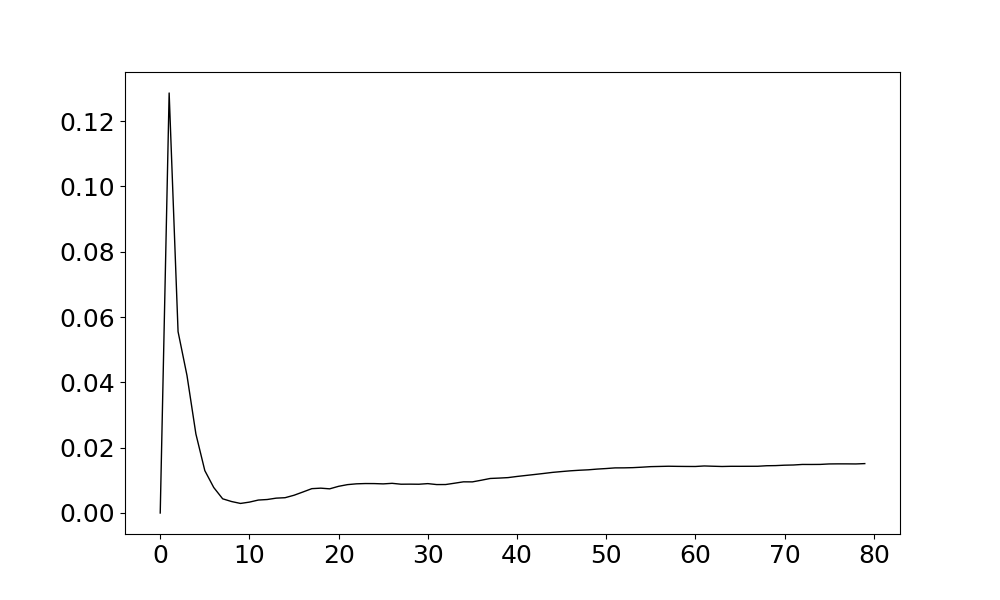}
      \caption*{$(\delta,\eta)=(10^{-3},10^{-4})$}
   \end{minipage}
\end{figure}

\begin{figure}[H]
   \centering
   \begin{minipage}[b]{0.3\textwidth}
      \includegraphics[width=5cm]{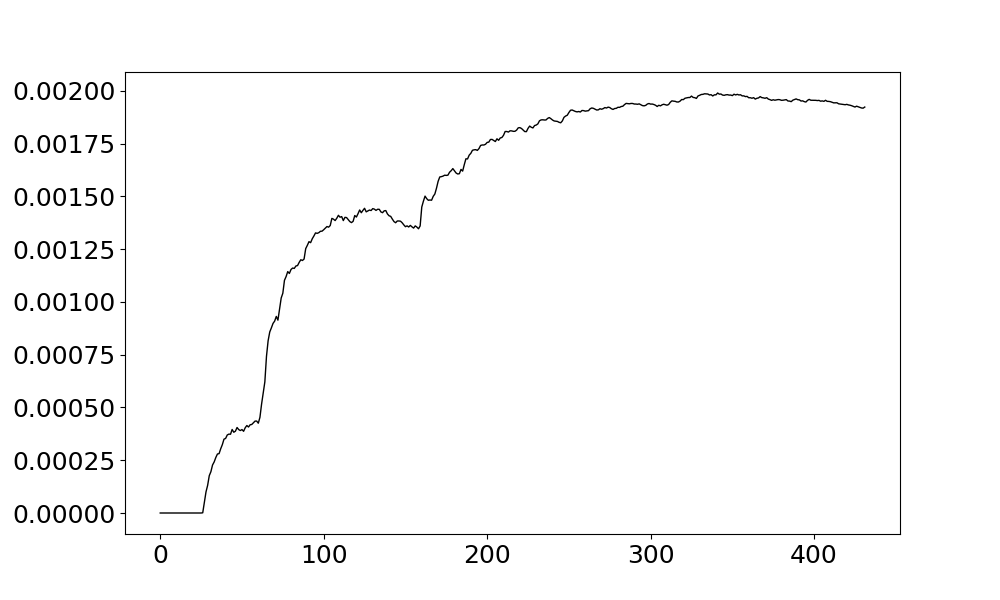}
      \caption*{$(\delta,\eta)=(10^{-4},10^{-2})$}
   \end{minipage}
   \h
   \begin{minipage}[b]{0.3\textwidth}
      \includegraphics[width=5cm]{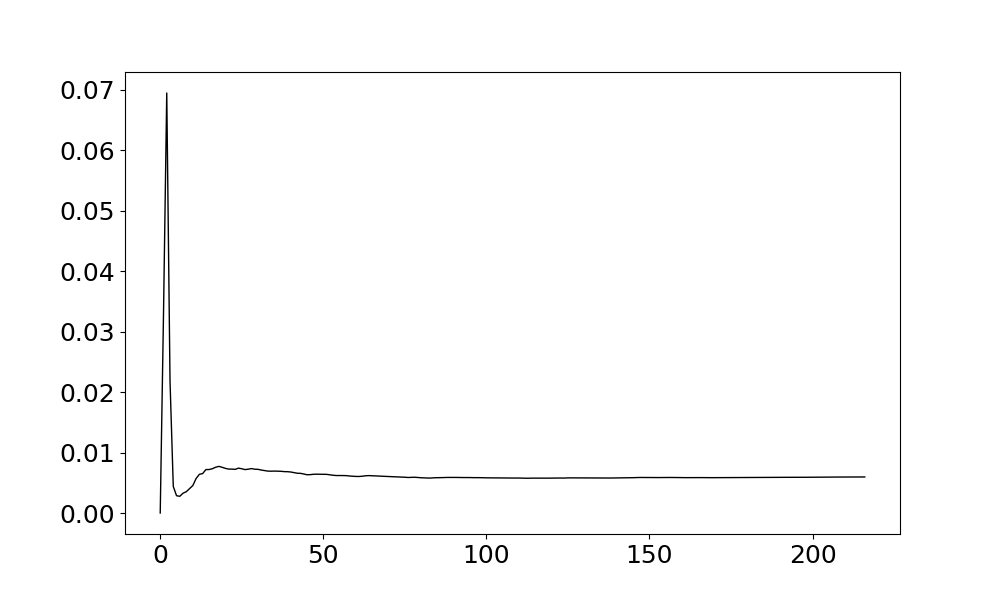}
      \caption*{$(\delta,\eta)=(10^{-4},10^{-3})$}
   \end{minipage}
   \h
   \begin{minipage}[b]{0.3\textwidth}
      \includegraphics[width=5cm]{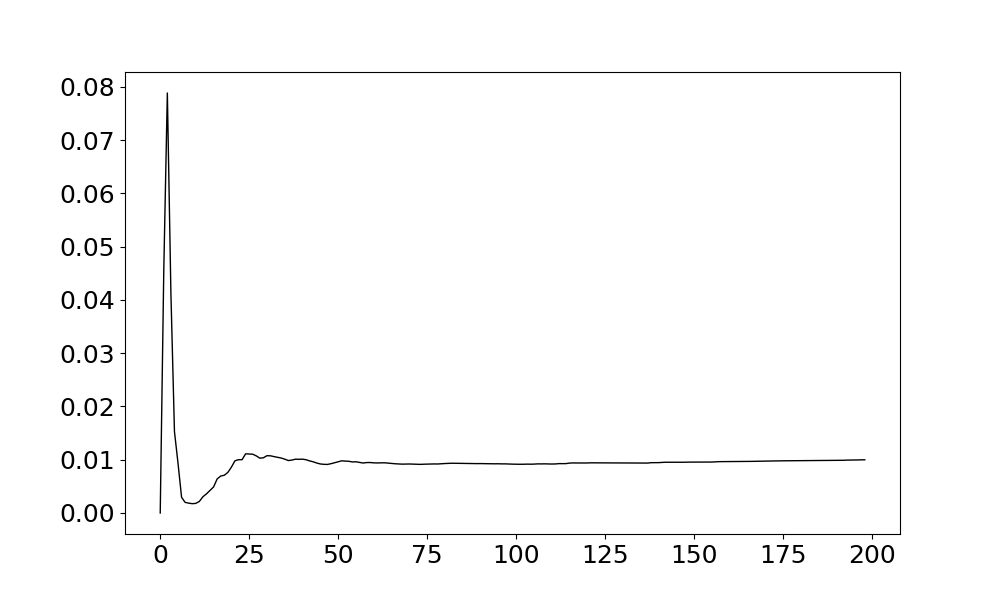}
      \caption*{$(\delta,\eta)=(10^{-4},10^{-4})$}
   \end{minipage}
\caption{Logarithmic relative error}\label{figerr1}
\end{figure}

\clearpage
\subsection{Mean compliance of linear elastic structures}\label{elasticresult}
We also exhibit the result in the setting of Subsection \ref{mincomp}.
We set $b=2.0$. We set the Lam\'e's constants $\lambda_1=\frac{15}{26}$, $\lambda_2=\frac{5}{13}$.
The design domain is set to $D=[0,2]\times [0,1]$ and
\begin{align*}
\Gamma_0&=\{0\} \times [0,1],\\
\Gamma_1&=\{1\} \times [0.5-0.06,0.5+0.06],\\
\Gamma_2&=\partial D \backslash (\Gamma_0 \cup \Gamma_1).
\end{align*}
We set the body force in $D$ to $f=0$ and the traction on $\Gamma_1$ to
\begin{align*}
g=
\begin{bmatrix}
 0 \\ -1 
\end{bmatrix}
\end{align*}
We set the initial mass $\mu_0=2dx$. We remark that $\mu_0$ is not a probability measure but we can derive the gradient flow just by replacing the functional $\cJ(\rho)$ with $\cJ(2\rho)$.
Figure \ref{figmass2} below show the results.
\begin{figure}[H]
   \centering
   \begin{minipage}[b]{0.3\textwidth}
      \includegraphics[width=5cm]{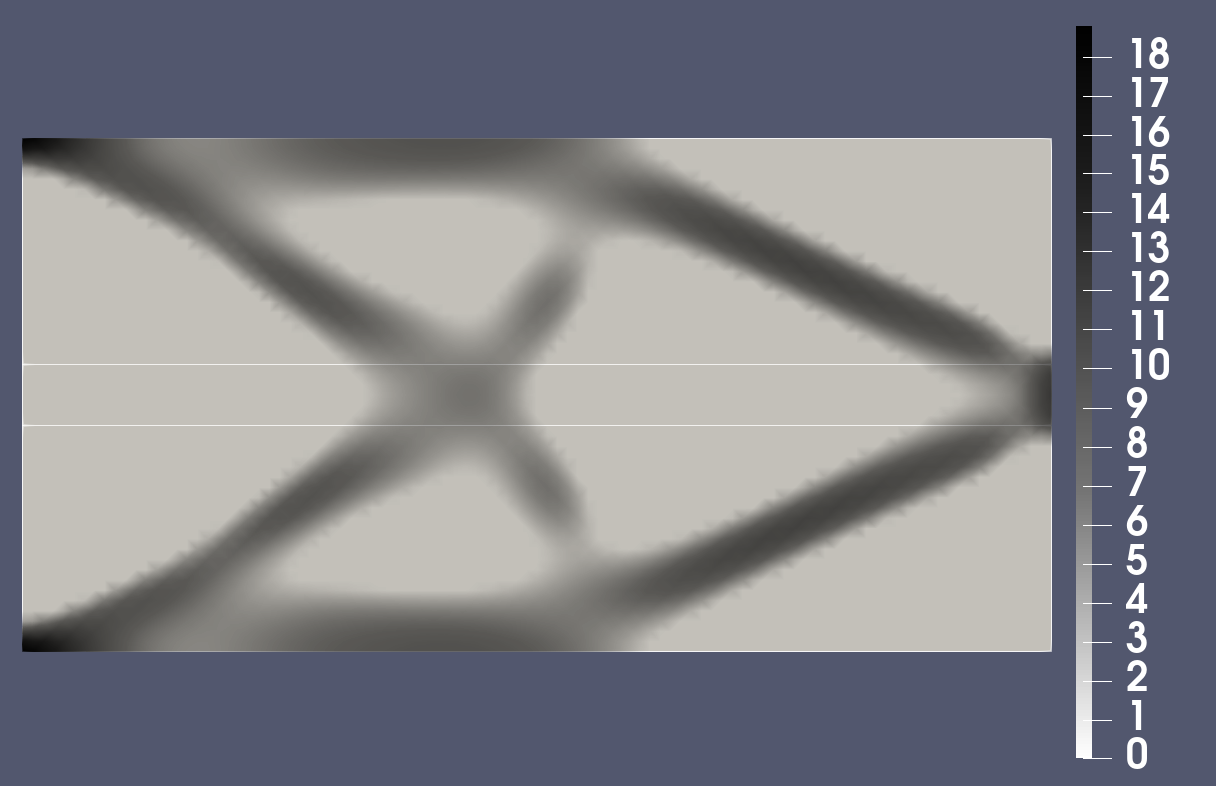}
      \caption*{$(\delta,\eta)=(10^{-2},10^{-2})$}
   \end{minipage}
   \h
   \begin{minipage}[b]{0.3\textwidth}
      \includegraphics[width=5cm]{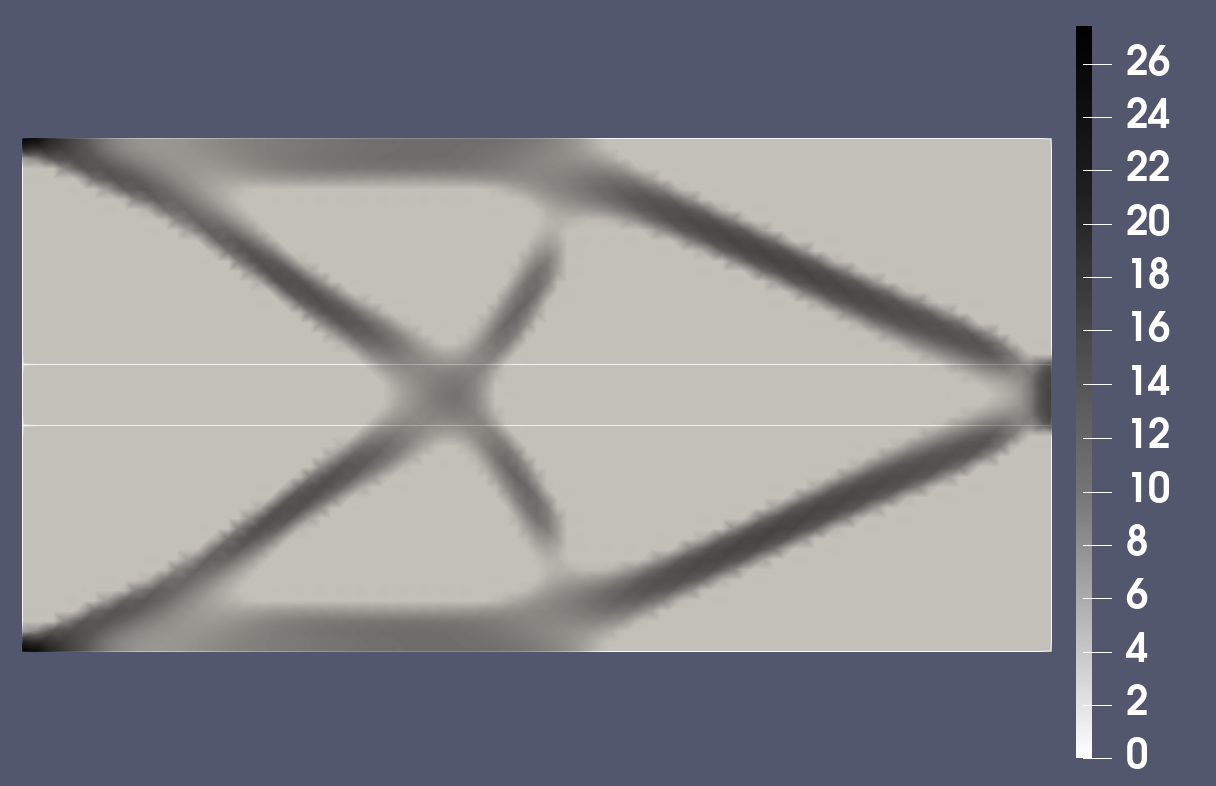}
      \caption*{$(\delta,\eta)=(10^{-2},10^{-3})$}
   \end{minipage}
   \h
   \begin{minipage}[b]{0.3\textwidth}
      \includegraphics[width=5cm]{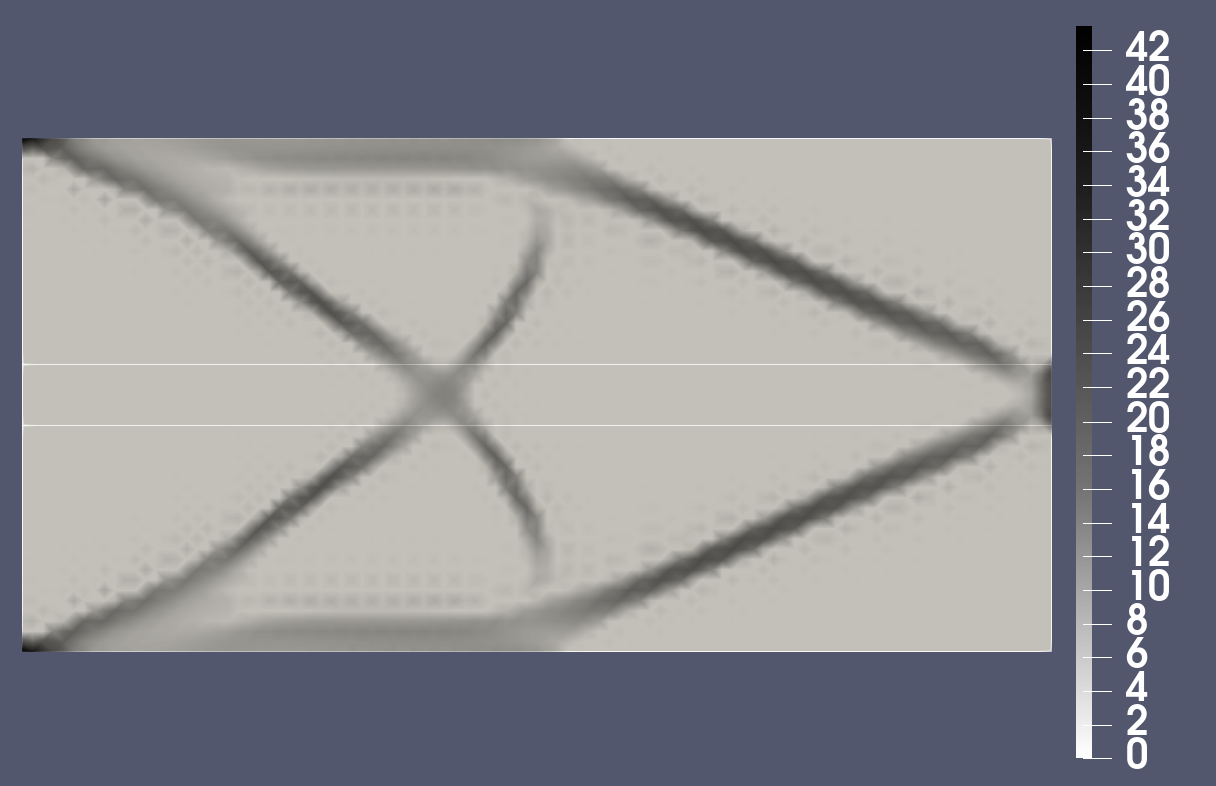}
      \caption*{$(\delta,\eta)=(10^{-2},10^{-4})$}
   \end{minipage}
\end{figure}

\begin{figure}[H]
   \centering
   \begin{minipage}[b]{0.3\textwidth}
      \includegraphics[width=5cm]{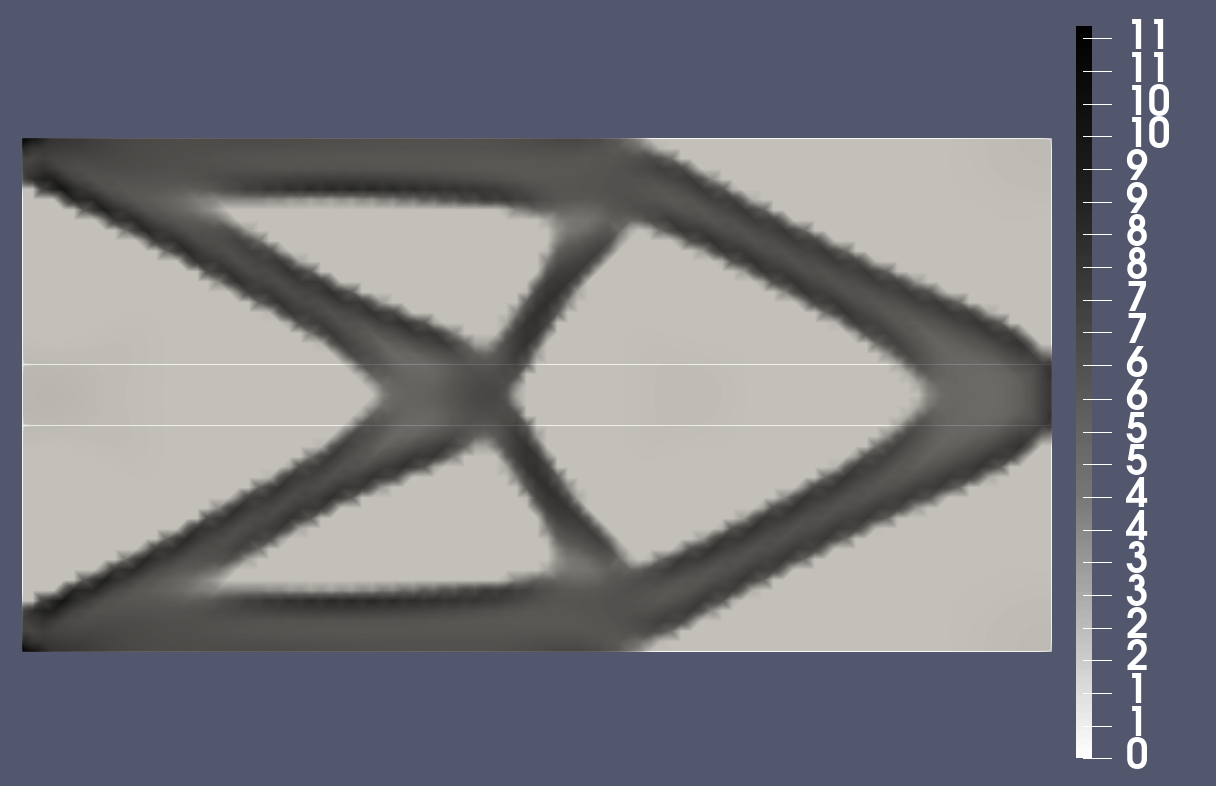}
      \caption*{$(\delta,\eta)=(10^{-3},10^{-2})$}
   \end{minipage}
   \h
   \begin{minipage}[b]{0.3\textwidth}
      \includegraphics[width=5cm]{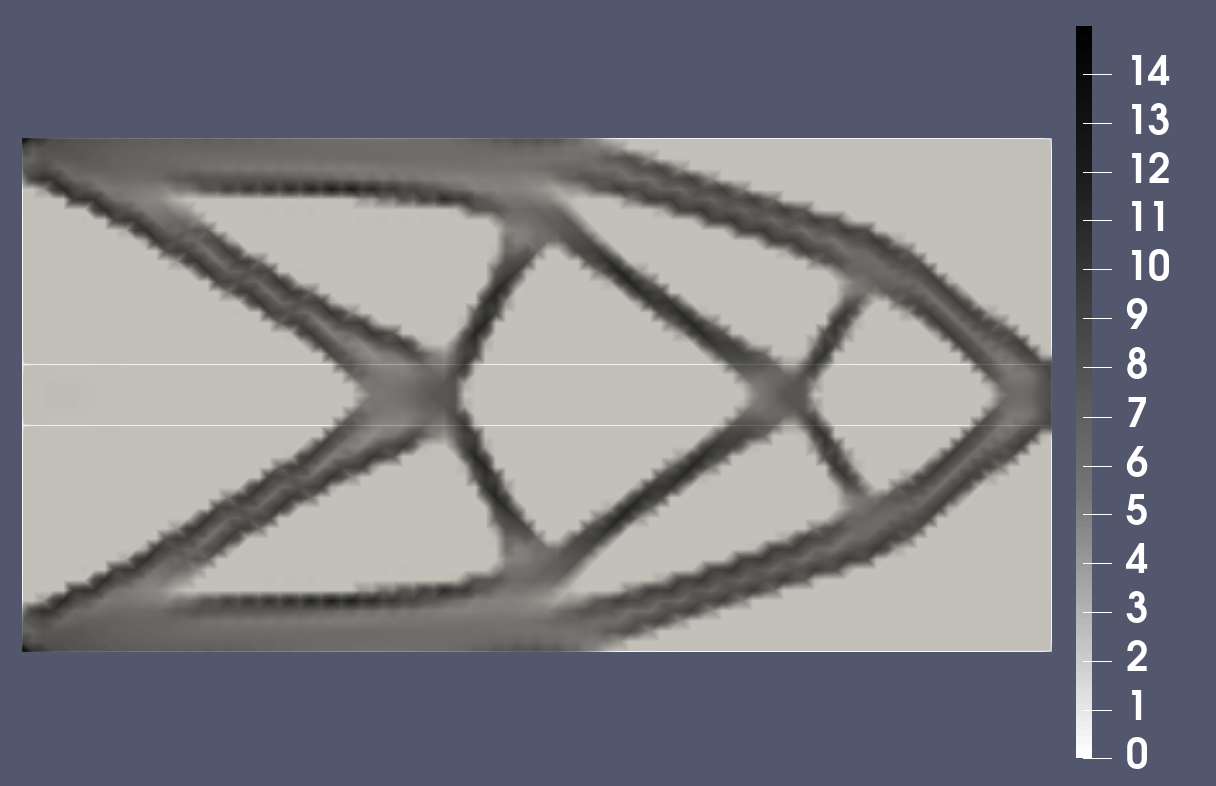}
      \caption*{$(\delta,\eta)=(10^{-3},10^{-3})$}
   \end{minipage}
   \h
   \begin{minipage}[b]{0.3\textwidth}
      \includegraphics[width=5cm]{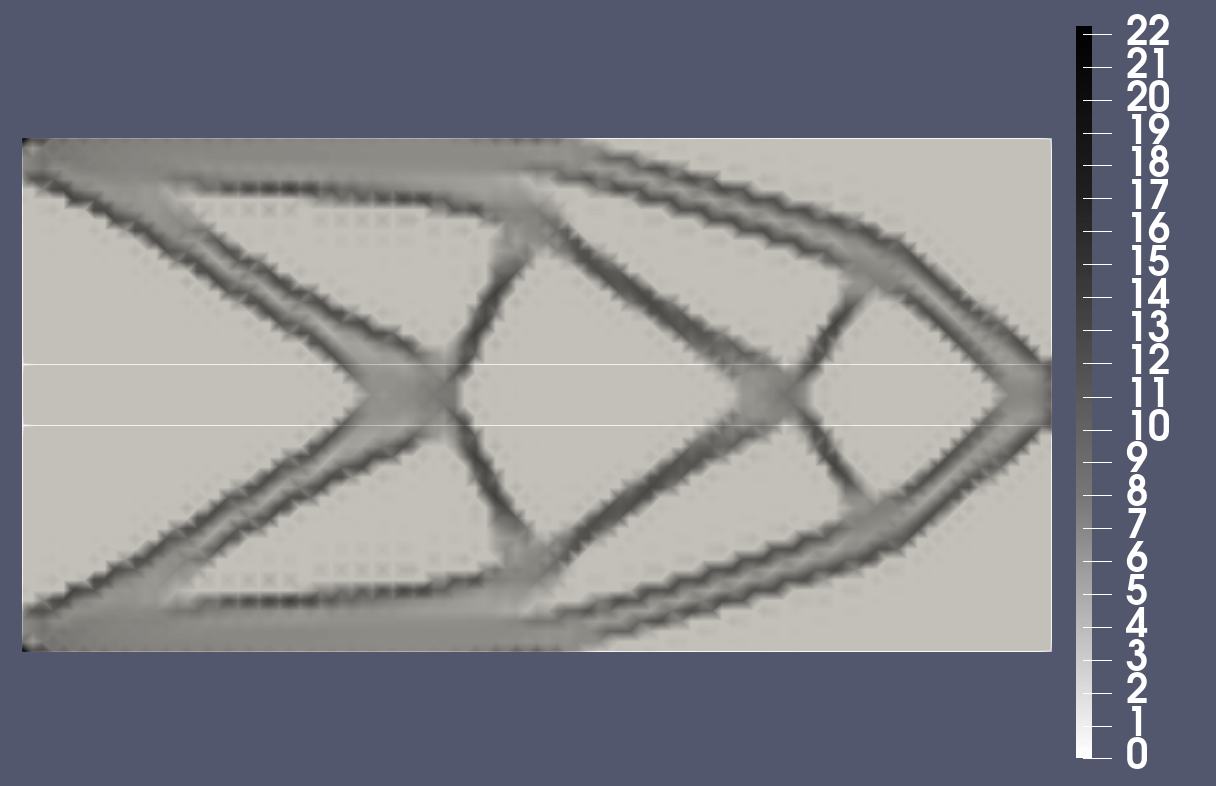}
      \caption*{$(\delta,\eta)=(10^{-3},10^{-4})$}
   \end{minipage}
\end{figure}

\begin{figure}[H]
   \centering
   \begin{minipage}[b]{0.3\textwidth}
      \includegraphics[width=5cm]{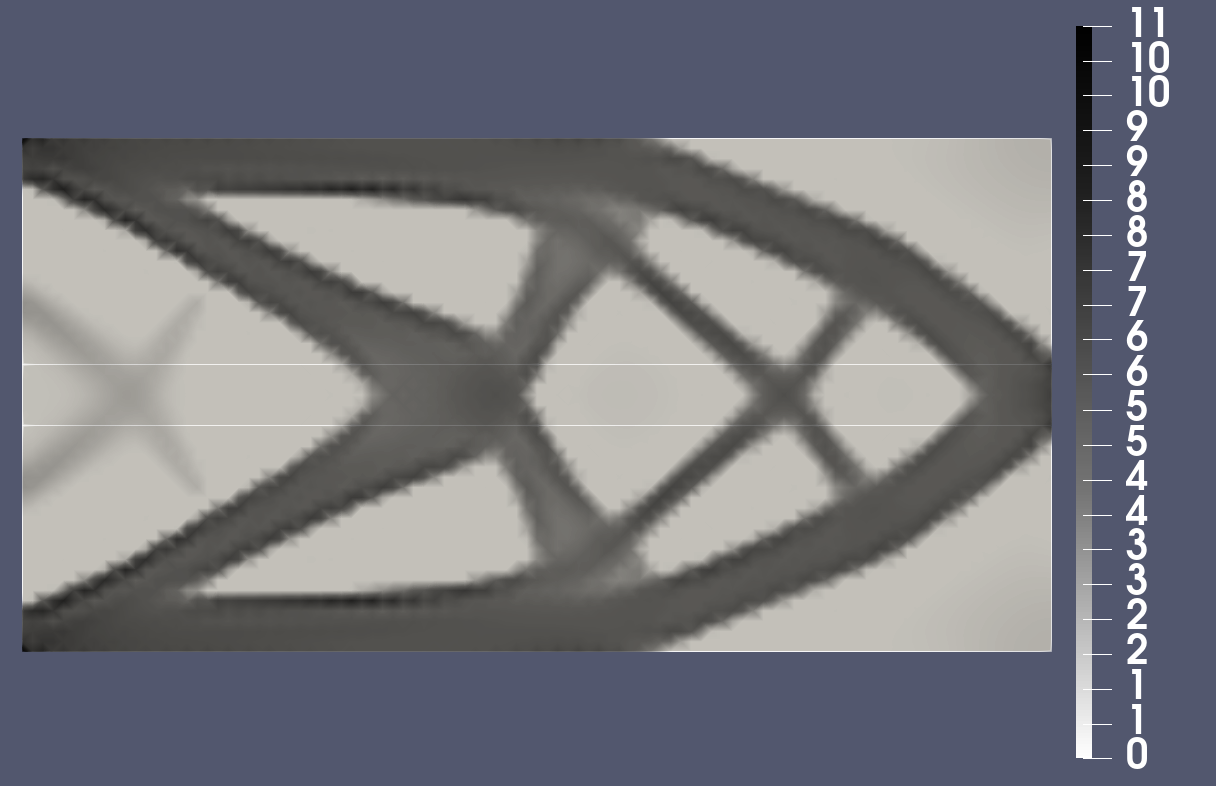}
      \caption*{$(\delta,\eta)=(10^{-4},10^{-2})$}
   \end{minipage}
   \h
   \begin{minipage}[b]{0.3\textwidth}
      \includegraphics[width=5cm]{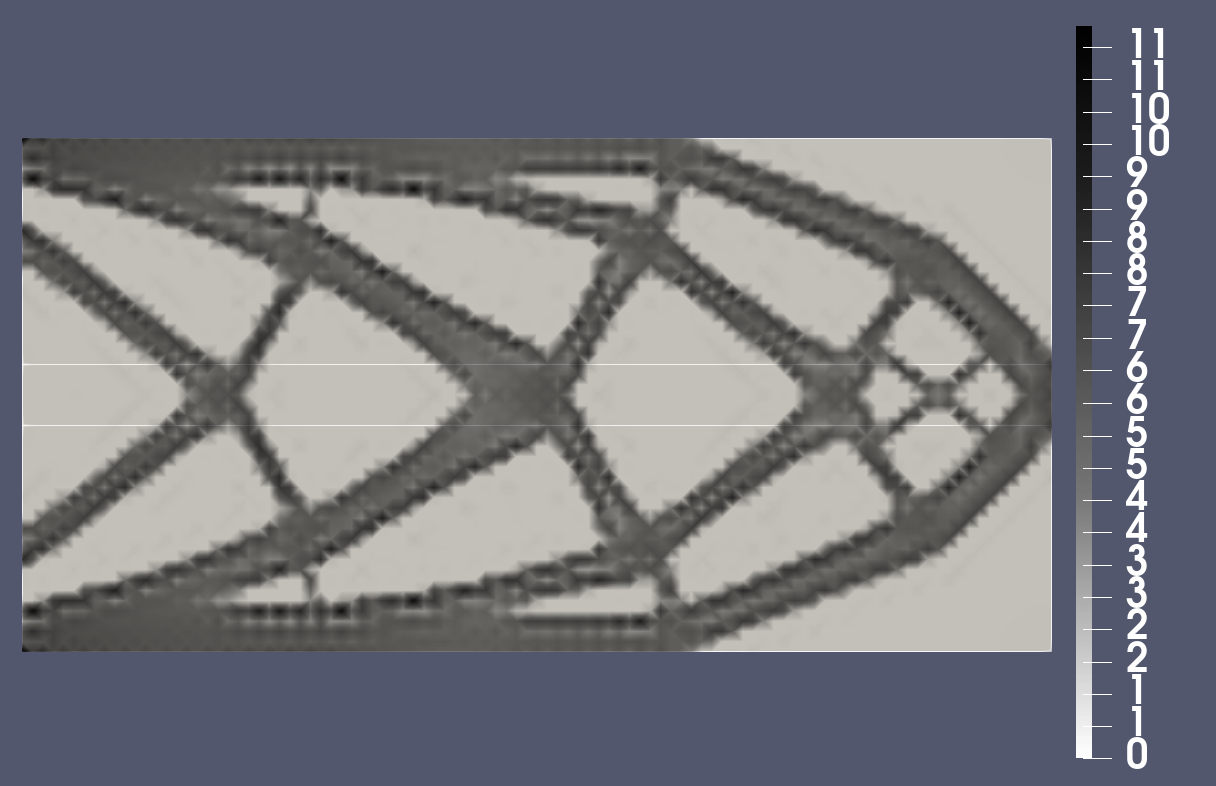}
      \caption*{$(\delta,\eta)=(10^{-4},10^{-3})$}
   \end{minipage}
   \h
   \begin{minipage}[b]{0.3\textwidth}
      \includegraphics[width=5cm]{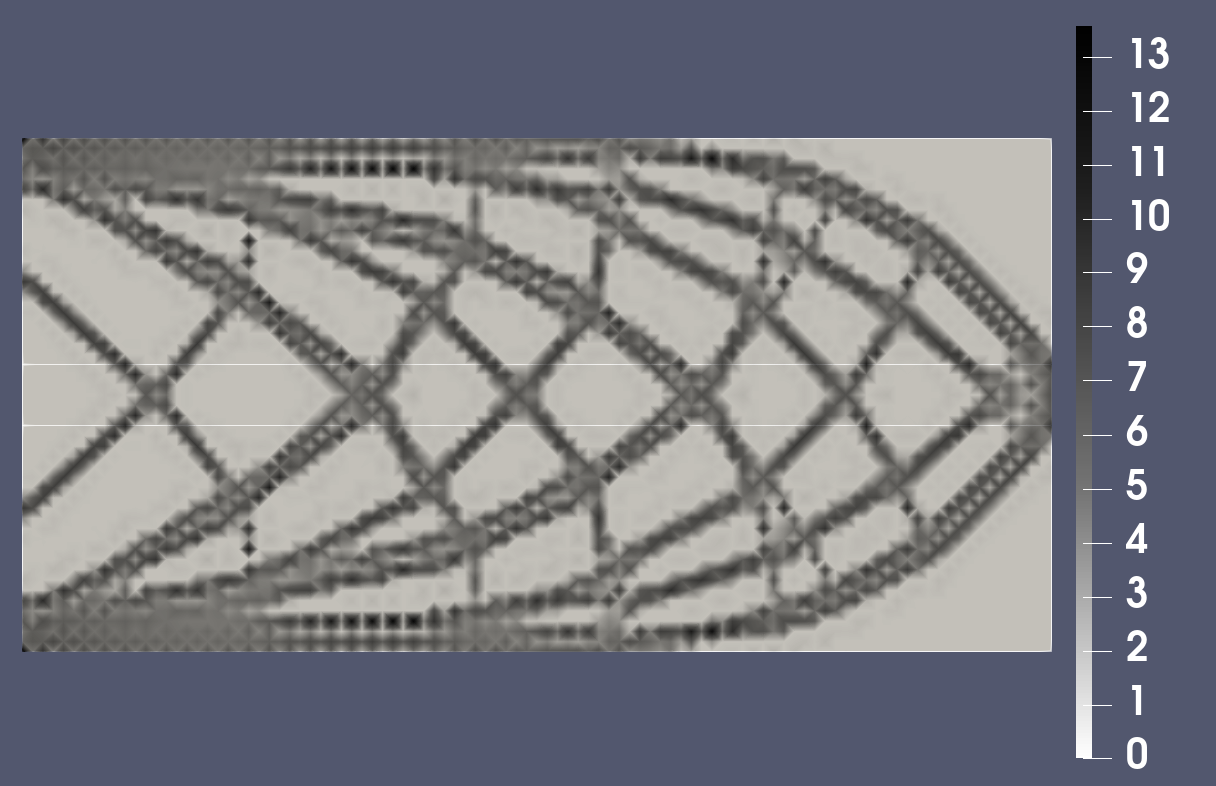}
      \caption*{$(\delta,\eta)=(10^{-4},10^{-4})$}
   \end{minipage}
\caption{Mass distribution}\label{figmass2}
\end{figure}
Figure \ref{figobj2} below shows how the objective function decreases. In this case, we set $\tau=3.0 \times 10^{-3}$ for $(\delta,\eta)=(10^{-2},10^{-2}), (10^{-2},10^{-3}), (10^{-2},10^{-4}), (10^{-3},10^{-2}), (10^{-3},10^{-3}), (10^{-3},10^{-4})$, $\tau=3.0 \times 10^{-4}$ for $(\delta, \eta)=(10^{-4},10^{-2}), (10^{-4},10^{-4})$ and $\tau=1.0 \times 10^{-3}$ for $(\delta,\eta)=(10^{-4},10^{-2})$.
\begin{figure}[H]
   \centering
   \begin{minipage}[b]{0.3\textwidth}
      \includegraphics[width=5cm]{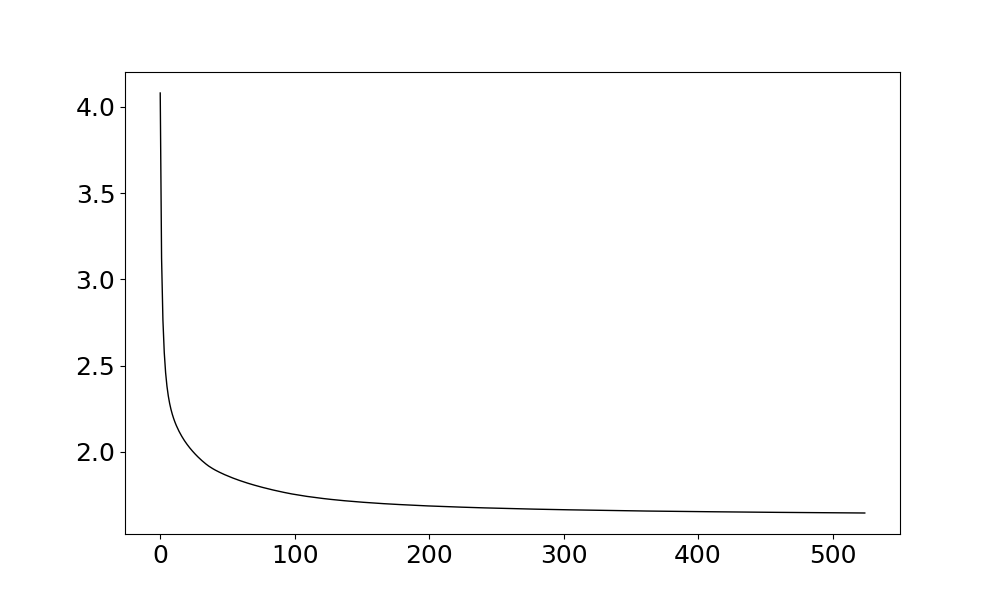}
      \caption*{$(\delta,\eta)=(10^{-2},10^{-2})$}
   \end{minipage}
   \h
   \begin{minipage}[b]{0.3\textwidth}
      \includegraphics[width=5cm]{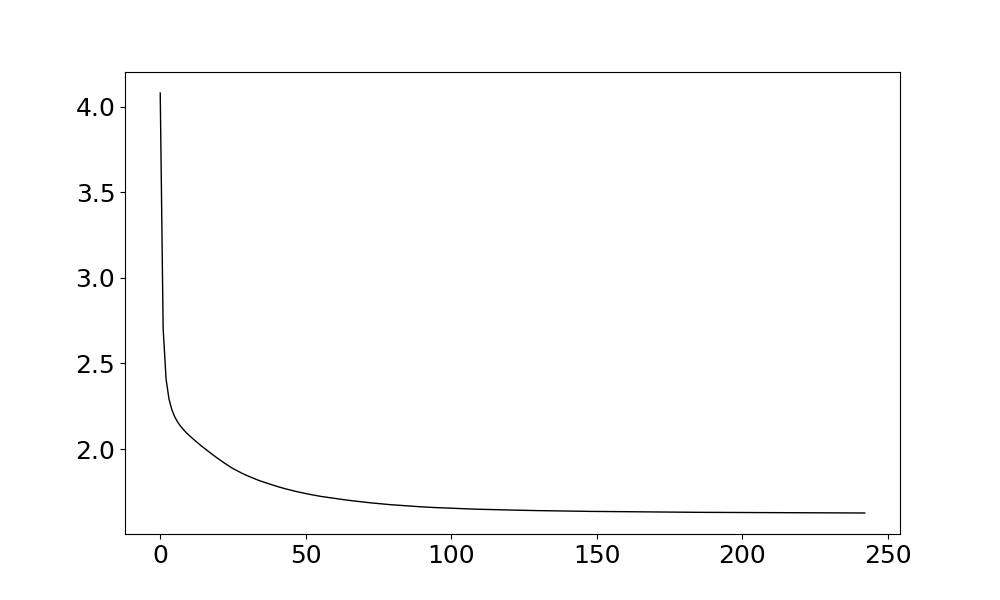}
      \caption*{$(\delta,\eta)=(10^{-2},10^{-3})$}
   \end{minipage}
   \h
   \begin{minipage}[b]{0.3\textwidth}
      \includegraphics[width=5cm]{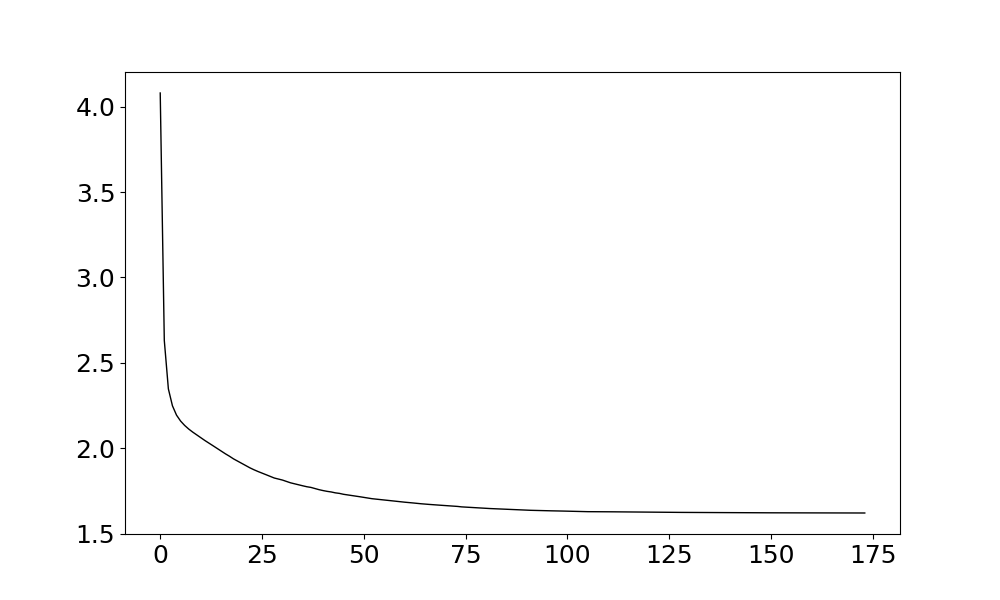}
      \caption*{$(\delta,\eta)=(10^{-2},10^{-4})$}
   \end{minipage}
\end{figure}

\begin{figure}[H]
   \centering
   \begin{minipage}[b]{0.3\textwidth}
      \includegraphics[width=5cm]{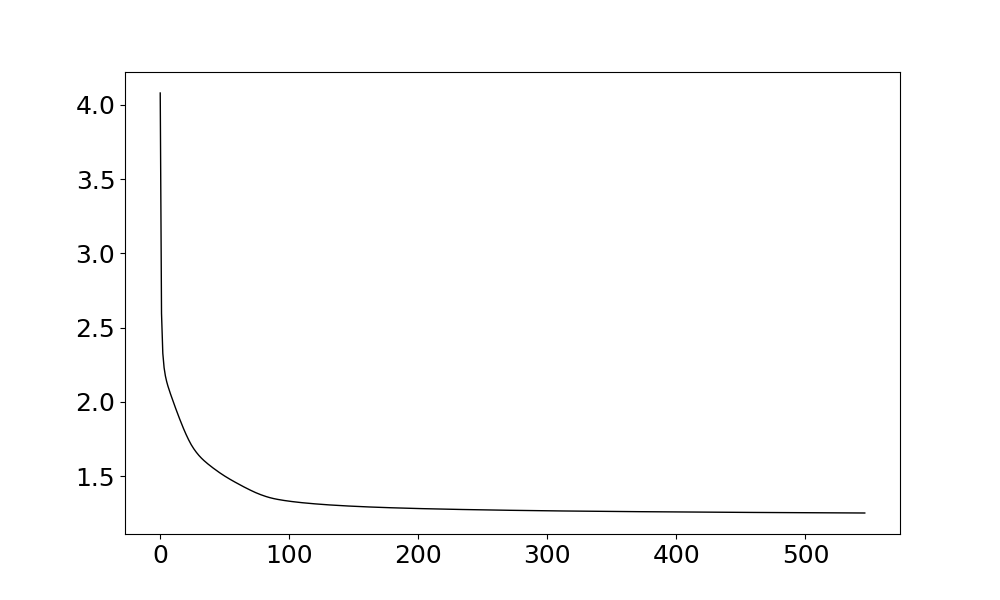}
      \caption*{$(\delta,\eta)=(10^{-3},10^{-2})$}
   \end{minipage}
   \h
   \begin{minipage}[b]{0.3\textwidth}
      \includegraphics[width=5cm]{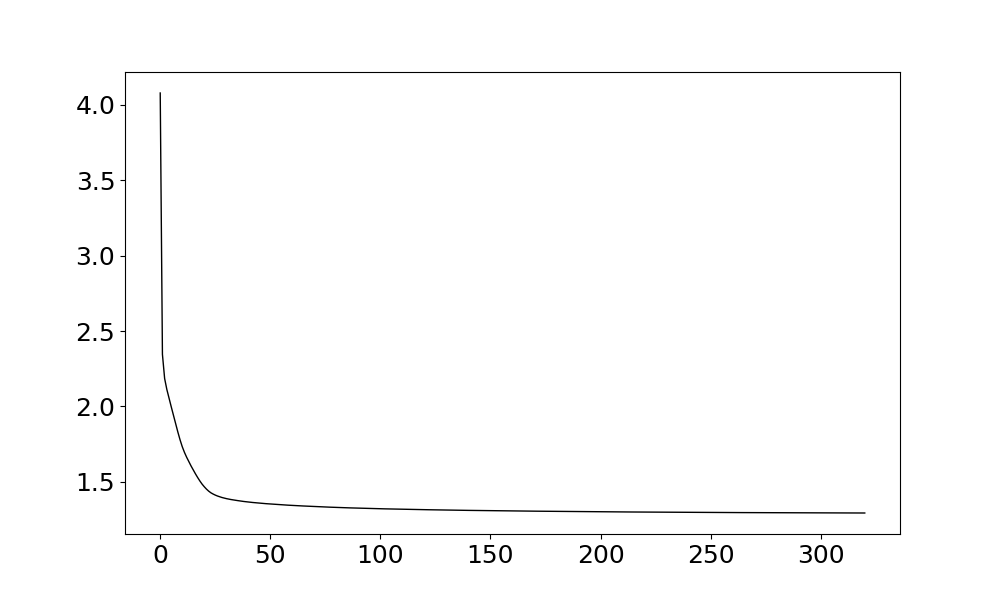}
      \caption*{$(\delta,\eta)=(10^{-3},10^{-3})$}
   \end{minipage}
   \h
   \begin{minipage}[b]{0.3\textwidth}
      \includegraphics[width=5cm]{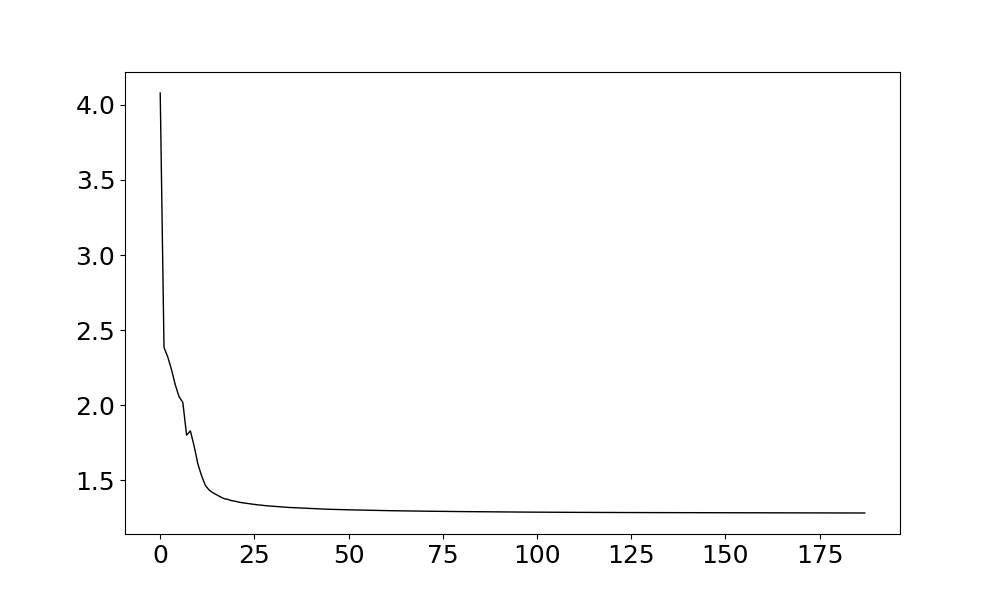}
      \caption*{$(\delta,\eta)=(10^{-3},10^{-4})$}
   \end{minipage}
\end{figure}

\begin{figure}[H]
   \centering
   \begin{minipage}[b]{0.3\textwidth}
      \includegraphics[width=5cm]{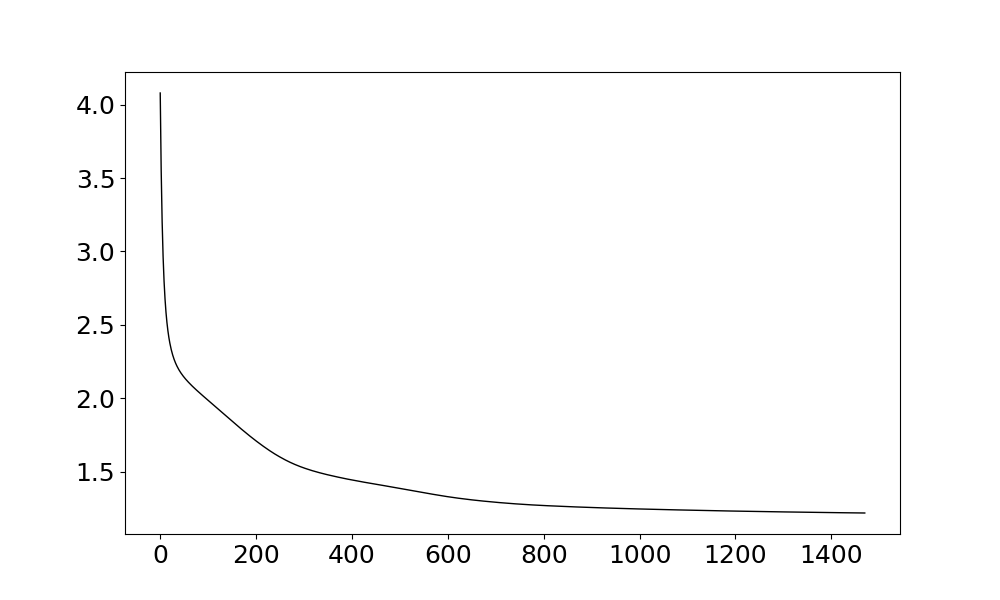}
      \caption*{$(\delta,\eta)=(10^{-4},10^{-2})$}
   \end{minipage}
   \h
   \begin{minipage}[b]{0.3\textwidth}
      \includegraphics[width=5cm]{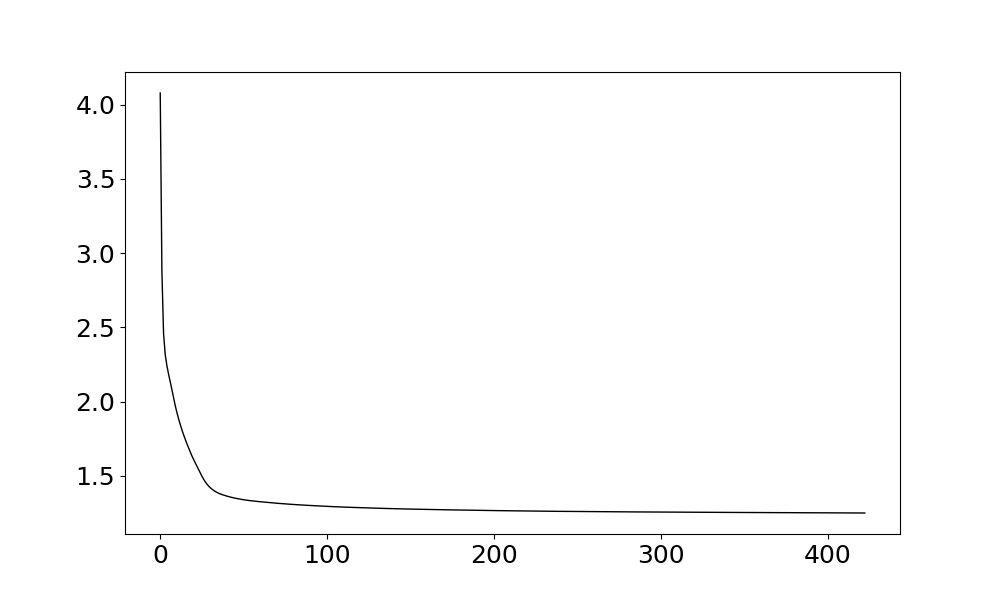}
      \caption*{$(\delta,\eta)=(10^{-4},10^{-3})$}
   \end{minipage}
   \h
   \begin{minipage}[b]{0.3\textwidth}
      \includegraphics[width=5cm]{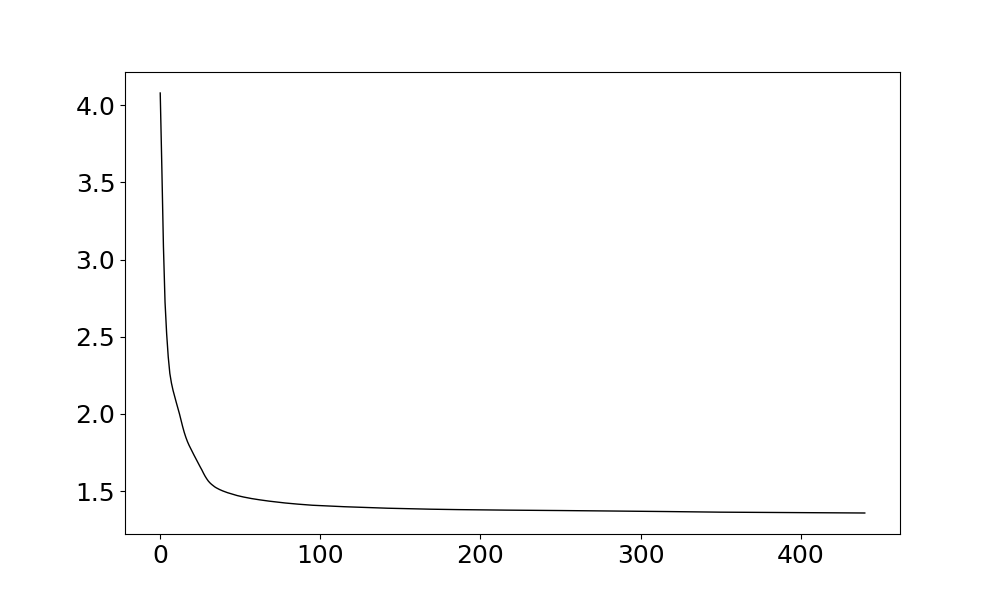}
      \caption*{$(\delta,\eta)=(10^{-4},10^{-4})$}
   \end{minipage}
\caption{Convergence history of the objective function}\label{figobj2}
\end{figure}
In the same way as Subsection \ref{heatresult}, Figure \ref{figerr2} descrives the logarithm of the relative error with respect to the initial mass.
\begin{figure}[H]
   \centering
   \begin{minipage}[b]{0.3\textwidth}
      \includegraphics[width=5cm]{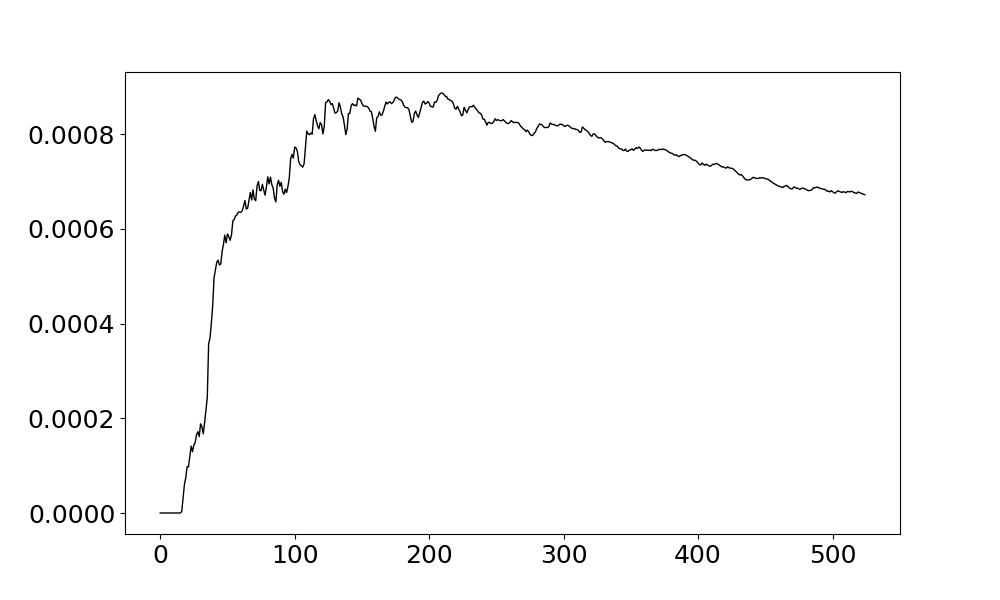}
      \caption*{$(\delta,\eta)=(10^{-2},10^{-2})$}
   \end{minipage}
   \h
   \begin{minipage}[b]{0.3\textwidth}
      \includegraphics[width=5cm]{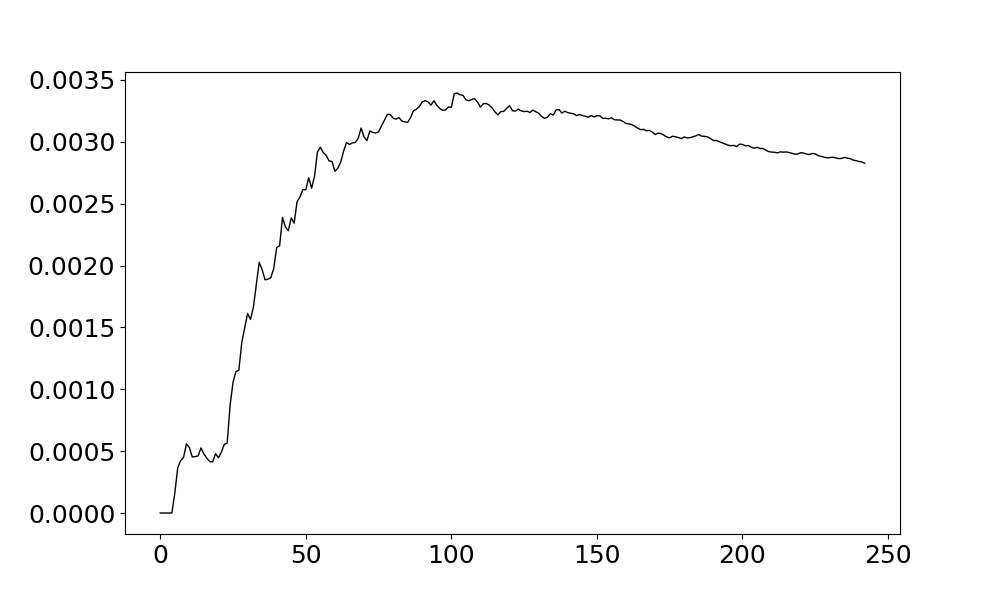}
      \caption*{$(\delta,\eta)=(10^{-2},10^{-3})$}
   \end{minipage}
   \h
   \begin{minipage}[b]{0.3\textwidth}
      \includegraphics[width=5cm]{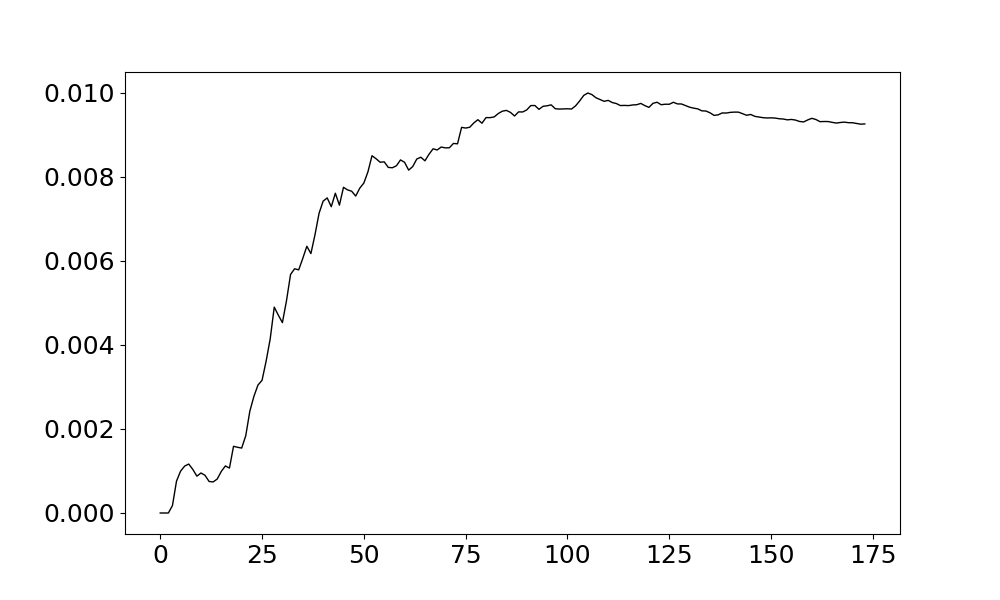}
      \caption*{$(\delta,\eta)=(10^{-2},10^{-4})$}
   \end{minipage}
\end{figure}

\begin{figure}[H]
   \centering
   \begin{minipage}[b]{0.3\textwidth}
      \includegraphics[width=5cm]{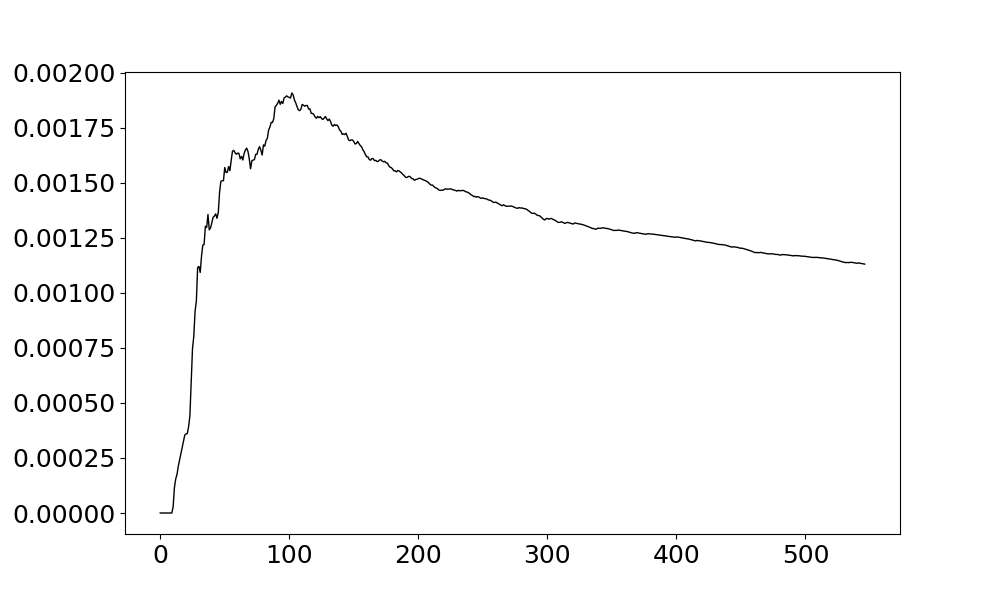}
      \caption*{$(\delta,\eta)=(10^{-3},10^{-2})$}
   \end{minipage}
   \h
   \begin{minipage}[b]{0.3\textwidth}
      \includegraphics[width=5cm]{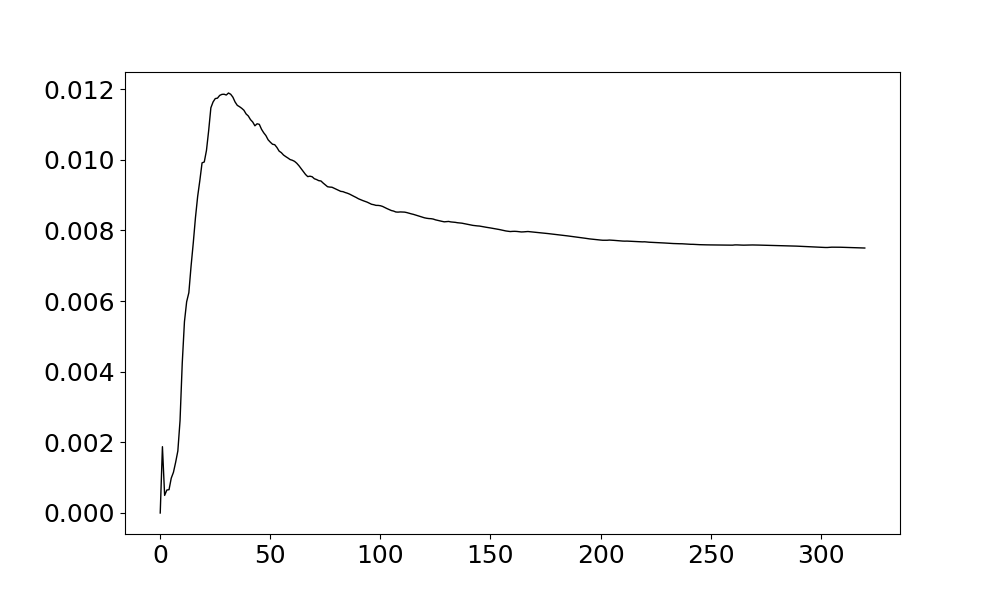}
      \caption*{$(\delta,\eta)=(10^{-3},10^{-3})$}
   \end{minipage}
   \h
   \begin{minipage}[b]{0.3\textwidth}
      \includegraphics[width=5cm]{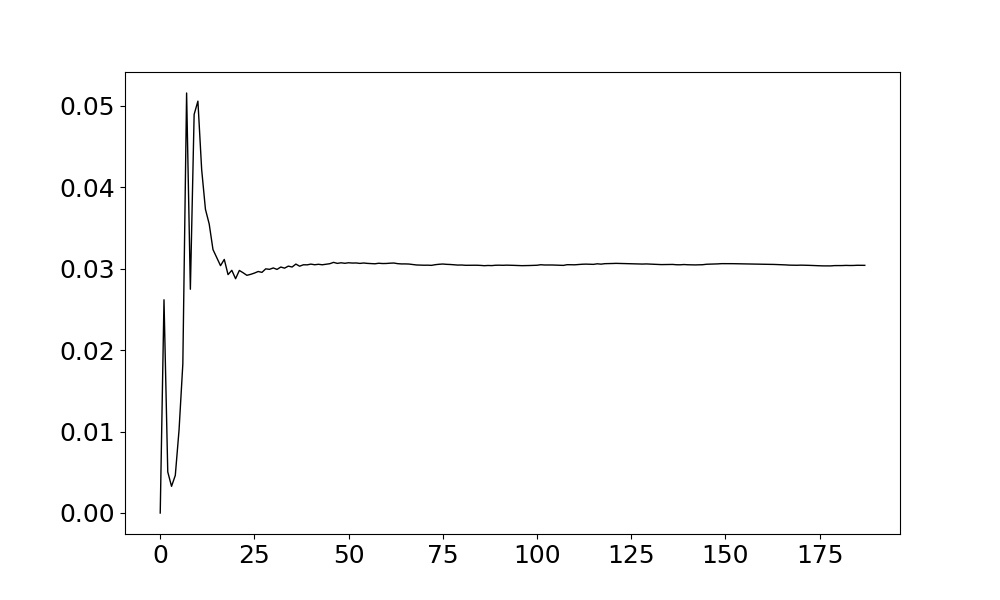}
      \caption*{$(\delta,\eta)=(10^{-3},10^{-4})$}
   \end{minipage}
\end{figure}

\begin{figure}[H]
   \centering
   \begin{minipage}[b]{0.3\textwidth}
      \includegraphics[width=5cm]{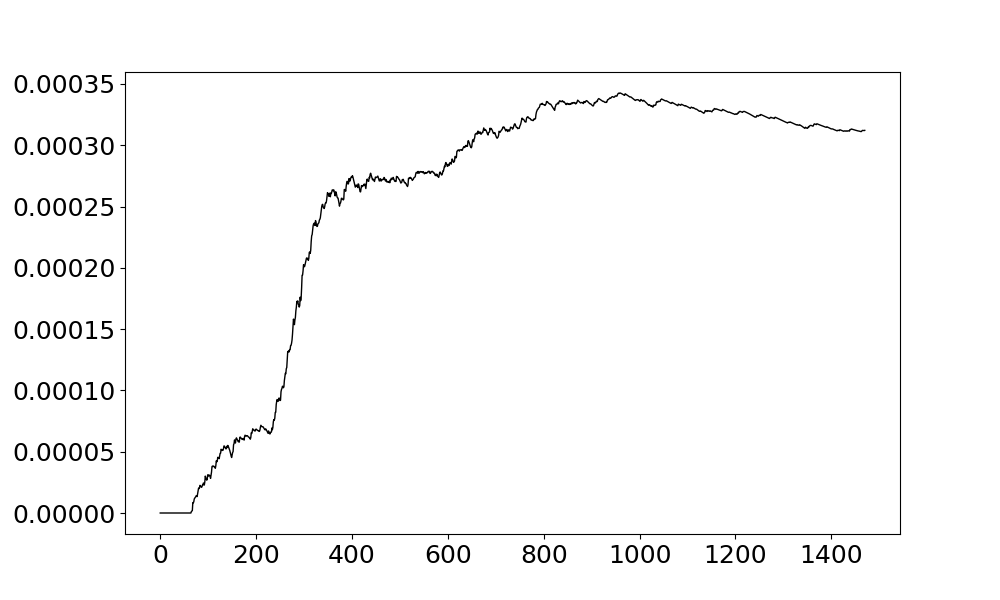}
      \caption*{$(\delta,\eta)=(10^{-4},10^{-2})$}
   \end{minipage}
   \h
   \begin{minipage}[b]{0.3\textwidth}
      \includegraphics[width=5cm]{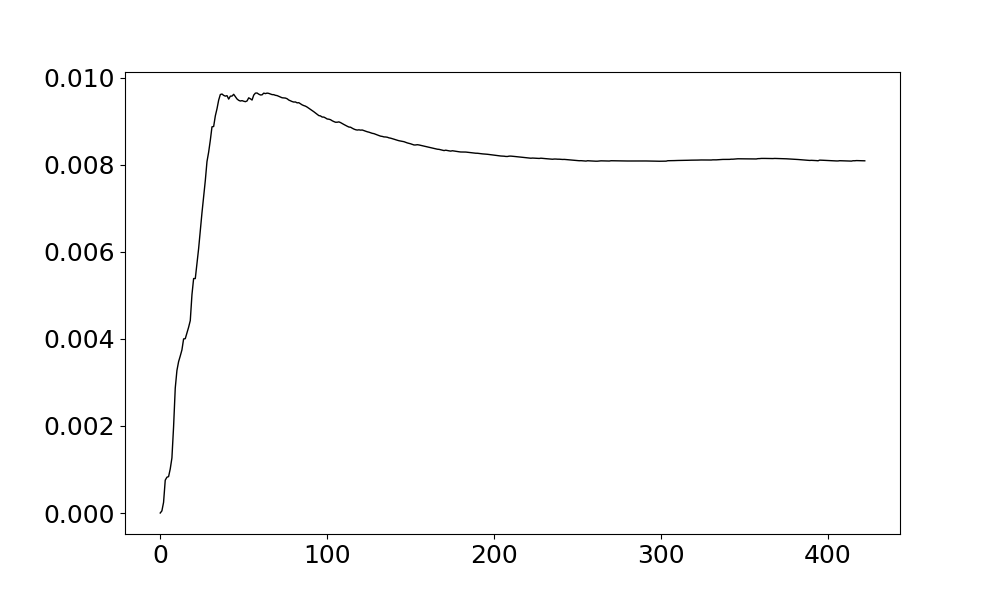}
      \caption*{$(\delta,\eta)=(10^{-4},10^{-3})$}
   \end{minipage}
   \h
   \begin{minipage}[b]{0.3\textwidth}
      \includegraphics[width=5cm]{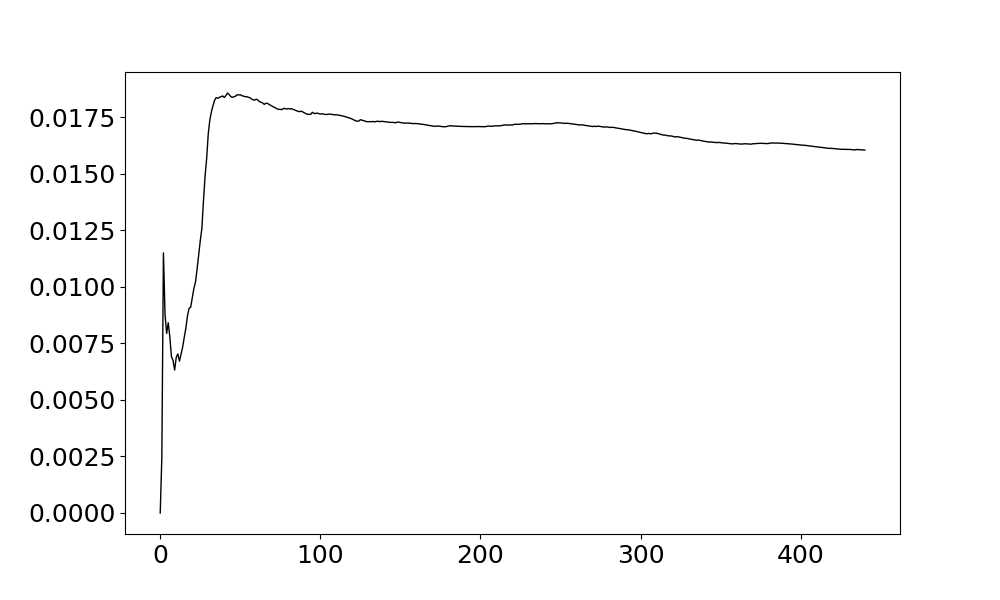}
      \caption*{$(\delta,\eta)=(10^{-4},10^{-4})$}
   \end{minipage}
\caption{Logarithmic relative error}\label{figerr2}
\end{figure}

\section*{CRediT authorship contribution statement}
F. Okazaki: Conceptualization, Software, Formal Analysis, Investigation, Writing -- Original Draft, Writing -- Review \& Editing (Lead).\\
T. Yamada: Supervision, Writing -- Review \& Editing (Supporting).
\section*{Acknowledgements}
The authors are grateful to A. Nakayasu, K. Sakai and K. Matsushima for their valuable advice.
\section*{Statements and Declarations}
{\bf Competing interests:} No conflict is related to this article, and the author has no relevant financial or non-financial interests to disclose.\\
{\bf Ethical Approval:} Not applicable to this article.\\
{\bf Funding:} This work was supported by JSPS KAKENHI Grant Number 24K22827 and 25K17264.\\
{\bf Data Availability Statements:} Data will be made available on request.

\end{document}